\documentclass{amsart}
\usepackage{enumitem}
\usepackage{amssymb}
\usepackage{tikz-cd}
\usepackage{hyperref}

\newtheorem{theorem}{Theorem}[section]
\newtheorem{lemma}[theorem]{Lemma}
\newtheorem{proposition}[theorem]{Proposition}
\newtheorem{corollary}[theorem]{Corollary}

\theoremstyle{definition}
\newtheorem{definition}[theorem]{Definition}
\newtheorem{example}[theorem]{Example}

\numberwithin{equation}{section}

\DeclareMathOperator{\supp}{supp}

\begin{document}

\title{Hardy inequalities in normal form}


%    author one information
\author{Gord Sinnamon}
\address{Department of Mathematics, Western University}
\curraddr{}
\email{sinnamon@uwo.ca}
\thanks{This work was supported by the Natural Sciences and Engineering Research Council of Canada.}

%    \subjclass is required.

\date{September 1, 2021}

%\dedicatory{}

%    Abstract is required.
\begin{abstract} A simple normal form for Hardy operators is introduced that unifies and simplifies the theory of weighted Hardy inequalities. A straightforward transition to normal form is given that applies to the various Hardy operators and their duals, whether defined on Lebesgue spaces of sequences, of functions on the half-line, or of functions on $\mathbb R^n$ or more general metric spaces. This is done by introducing an abstract formulation of Hardy operators, more general than any of these, and showing that the normal form transition applies to all operators formulated in this way.

The transition to normal form is shown to preserve boundedness, compactness, and operator norm. To a large extent the transition can be carried out via well-behaved linear operators.

Known results for boundedness and compactness of Hardy operators are given simple proofs and extended, via the transition, to this general setting. 

New estimates for the best constant in Hardy inequalities are established and a large class of Hardy inequalities is identified in which the best constants are known precisely.
\end{abstract}

\maketitle

\section{Introduction: Hardy Operators}

Hardy's integral operator,  $\int_0^x f(t)\,dt$  and averaging operator,  $\frac1x\int_0^x f(t)\,dt$, have been extensively studied since their introduction in 1915, not only due the importance of the operators in applications, but because they constitute simple exemplars of positive integral operators. Indeed, techniques introduced to study Hardy operators and their duals, $\int_t^\infty f(x)\,dx$ and $\int_t^\infty f(x)\,\frac{dx}x$, have been successfully applied in the analysis of a wide variety of related integral operators and differential operators.

The discrete Hardy operators, $\sum_{k=1}^na_k$ and $\frac1k\sum_{k=1}^na_k$ (and their duals) have also been studied, but despite the obvious similarity to the integral and averaging operators, the continuous and discrete theories developed almost separately. Results for Hardy operators on $\mathbb R$ with general measures include those above, but have not really succeeded in uniting the two theories. See \cite{KMP07} for a historical perspective. Hardy operators involving functions defined on $\mathbb R^n$ or on more general metric spaces, but retaining their one-dimensional character, have been studied using a polar-coordinate approach. See, for example, \cite{S1D,CG,Sgradient,RV19,RV21}. All of the above are included in the normal-form approach to Hardy operators introduced here. 

The main goal of this paper is threefold: To introduce a very general class of Hardy operators called \emph{abstract Hardy operators}; to introduce a very restricted class of Hardy operators, called \emph{normal form Hardy operators}; and to show that each abstract Hardy operator corresponds in a natural way to a normal form Hardy operator that shares important properties of the original, including boundedness, compactness and operator norm. In this way, we provide a framework to enable work on a simple, restricted class of Hardy operators to yield results for the classes of Hardy operators mentioned above and more. One immediate advantage is that technical assumptions imposed in the polar-coordinate approach can be removed.

A second goal is to demonstrate the advantages of working with normal form Hardy operators by giving simple proofs of known facts and proving new results inspired by the simplicity of normal form. The richest source of information about the various Hardy operators is the characterization of those Lebesgue spaces (of sequences or functions, equipped with weights or general measures) between which the operator acts as a bounded map. The boundedness is expressed as a so-called \emph{weighted Hardy inequality}. We contend that this information can be obtained more readily by a change in point of view that is implicit in the class of normal form Hardy inequalities. Instead of fixing a Hardy operator and allowing the varying domain and codomain spaces to generate weighted Hardy inequalities, all normal form Hardy operators are maps from (unweighted) $L^p(0,\infty)$ to (unweighted) $L^q(0,\infty)$. The corresponding weighted Hardy inequalities are then generated by varying a single parameter. The parameters are the decreasing functions on the half line; a class that is very easy to describe and to work with. 

Among the many advantages of this point of view, two stand out: First, normal form facilitates direct comparison between different Hardy operators because they all act on the same class of functions. Second, both parameters and domain functions are readily approximated by well-behaved functions, which simplifies arguments and reduces technicalities.

This introduction concludes by recalling some standard notation and well-known tools. Then, in Section \ref{S2}, both abstract and normal form Hardy operators are defined and the basic relationship between them is stated. (It is proved in Section \ref{S5}). In addition, normal form parameters are computed for some commonly studied classes of weighted Hardy operators. 

Normal form Hardy inequalities are investigated in Section \ref{S3}. This includes simple proofs of the standard boundedness and compactness results, which generalize them to all abstract Hardy inequalities. It also includes a number of results, primarily about best constants, that hold in the abstract but appear to be new even for the standard weighted Hardy inequalities. See part $(iii)$ and the lower bound from part $(ii)$ in each of Theorems \ref{alt0}, \ref{alt1} and \ref{alt2}. See also Lemma \ref{tighter}.

In Section \ref{S4}, a notion of order in a general measure space is introduced and studied. Called an \emph{ordered core}, it is the underlying structure needed to establish, in Section \ref{S5}, the close connection between an abstract Hardy operator and its normal form. 

\subsection{Notation and basic tools}

If $(Y,\mu)$ is a $\sigma$-finite measure space, $L^+_\mu$ denotes the collection of measurable functions taking values in $[0,\infty]$. For $0<p\le\infty$, $L^p_\mu$ and $\|\cdot\|_{L^p_\mu}$ denote the usual Lebesgue space (of complex-valued functions) and its norm or quasi-norm. For Lebesgue measure on $(0,\infty)$, which we denote by $m$, these will be simplified to $L^+$, $L^p$ and $\|\cdot\|_p$. 

If $b:(0,\infty)\to[0,\infty]$ is non-increasing, its \emph{generalized inverse} $b^{-1}$ is defined by $b^{-1}(t)=\sup\{x>0:t<b(x)\}$, where $\sup\emptyset$ is taken to be zero. Clearly, $b^{-1}$ is also a non-increasing function that maps $(0,\infty)$ to $[0,\infty)$. In addition, $b^{-1}$ is right continuous and it is well known that if $b$ is right continuous, then $(b^{-1})^{-1}=b$;  for all $x>0$, $b^{-1}(b(x))\le x$; and for all $x,t>0$,
\begin{equation}\label{fub}
t<b(x)\quad\mbox{if and only if}\quad x<b^{-1}(t).
\end{equation}
The last equivalence, together with the Fubini/Tonelli theorem, shows that if $F(x,t)$ is non-negative and Lebesgue measurable on $(0,\infty)^2$, then 
\begin{equation}\label{fubton}
\int_0^\infty\int_0^{b(x)}F(x,t)\,dt\,dx=\int_0^\infty\int_0^{b^{-1}(t)}F(x,t)\,dx\,dt.
\end{equation}
Since any non-increasing $b$ agrees almost everywhere with a right-continuous one, the last equation holds without assuming right continuity.

If $(Y,\mu)$ is a $\sigma$-finite measure space and $f$ is a $\mu$-measurable function on $Y$, the \emph{distribution function} of $f$ is $\mu_f(t)=\mu\{y\in Y:|f(y)|>t\}$, for $t>0$. It is non-increasing and right continuous. The non-increasing rearrangement, denoted $f^*$, is the generalized inverse of $\mu_f$. Well-known properties of the rearrangement include: $\mu_f=m_{f^*}$; for $0<p\le \infty$, $\|f\|_{L^p_\mu}=\|f^*\|_p$; and
\begin{equation}\label{HL}
\int_Xfg\,d\mu\le\int_0^\infty f^*g^*.
\end{equation}
Also, if $f\in L^+_\mu$ and $\varphi\in L^+$ is non-decreasing and continuous, then $(\varphi\circ f)^*=\varphi\circ f^*$.

If $(S,\lambda)$ is another $\sigma$-finite measure space, then a linear map $T:L^1_\mu+L^\infty_\mu\to L^1_\lambda+L^\infty_\lambda$ is called \emph{admissible} if its restrictions $T:L^1_\mu\to L^1_\lambda$ and $T:L^\infty_\mu\to L^\infty_\lambda$ are bounded. The \emph{norm} of an admissible map is the larger of the norms of these two restrictions. An admissible map is \emph{positive} if it maps non-negative functions to non-negative functions. It is well known that if $1\le q\le \infty$, then every admissible map is bounded from $L^q_\mu$ to $L^q_\lambda$; if the admissible map has norm at most one then so does its restriction to $L^q_\mu$.

\section{Two Classes of Hardy Operators}\label{S2}

The operators we introduce below are defined as maps from positive functions to positive functions. They are additive and homogeneous for positive constants. When such an operator is bounded from the positive functions in one Banach function space to another Banach function space it automatically extends to a linear operator on (all functions in) the domain space with no increase in operator norm. We use the same symbol to denote the operator and its extension.

\subsection{Abstract Hardy Operators}

\begin{definition}\label{coremap}Let $(S,\Sigma,\lambda)$ and $(Y,\mu)$ be $\sigma$-finite measure spaces. We say $B:Y\to \Sigma$ is a \emph{core map} if it has the following properties:
\begin{enumerate}[label=\rm({\Roman*})]
\item\label{to} (Total orderedness) the range of $B$ is a totally ordered subset of $\Sigma$;
\item\label{me} (Measurability) for $E\in\Sigma$, $y\mapsto\lambda(E\cap B(y))$ is $\mu$-measurable;
\item\label{se} ($\sigma$-boundedness) $\cup_{y\in Y}B(y)=\cup_{y\in Y_0}B(y)$ for some countable $Y_0\subseteq Y$;
\item\label{fi} (Finiteness) for all $y\in Y$, $\lambda(B(y))<\infty$. 
\end{enumerate}
\end{definition}
\begin{definition} For a core map $B$, the \emph{abstract Hardy operator} is
\[
K_Bg(y)=\int_{B(y)}g\,d\lambda, \quad\mbox{for } g\in L^+_\lambda.
\]
For $p,q\in(0,\infty]$ and $C\ge0$, the associated \emph{abstract Hardy inequality} is
\[
\|K_Bg\|_{L^q_\mu}\le C\|g\|_{L^p_\lambda}, \quad\mbox{for all } g\in L^+_\lambda.
\]
\end{definition}
If the abstract Hardy inequality holds for some finite $C$, the operator norm of $K_B:L^p_\lambda\to L^q_\mu$ is the smallest such $C$. Since we are working with positive operators, it is convenient to take the operator norm of an unbounded operator to be $\infty$.

Total orderedness \ref{to} is the main restriction on $B$. The next two are technical: \ref{me} ensures the measurability of each $K_Bg$ and \ref{se} ensures $\sigma$-finiteness of the $\sigma$-algebra generated by the range of $B$. (Neither of these is likely to fail unless an effort is made to arrange its failure.) The finiteness property \ref{fi} may appear unduly restrictive but Theorem \ref{3meas} shows it is not.

For abstract Hardy operators the measure on the domain space is the same as that used to define the operator. When these two measures differ we have a so-called three-measure Hardy inequality. It is well known that when $p\ne 1$ this seemingly more general situation can be reduced to the two-measure case above. For abstract Hardy inequalities this reduction is done in Theorem \ref{3meas}.

\subsection{Normal Form Hardy Operators}

\begin{definition} Let $\mathcal B$ be the set of all non-increasing functions $b:(0,\infty)\to[0,\infty]$. Each $b\in \mathcal B$ will be called a \emph{normal form parameter} and the \emph{normal form Hardy operator} $H_b$ is defined by 
\[
H_bf(x)=\int_0^{b(x)}f(t)\,dt, \quad\mbox{for all } f\in L^+.
\]
For $p,q\in(0,\infty]$ and $C\ge0$, the associated \emph{normal form Hardy inequality} is
\[
\|H_bf\|_q\le C\|f\|_p, \quad\mbox{for all } f\in L^+.
\]
The smallest constant $C$, finite or infinite, for which the inequality holds will be denoted $N_{p,q}(b)$. It is the operator norm of $H_b$ as a map from $L^p$ to $L^q$. In a common abuse of notation we will sometimes replace the function name $b$ in $N_{p,q}(b)$ by a formula in the variable $x$.

Each $b\in\mathcal B$ agrees almost everywhere with a right-continuous function in $\mathcal B$. For each $f\in L^+$ this leaves $H_bf$ essentially unchanged. For this reason we may assume, when convenient, that normal form parameters are right continuous. 

\end{definition}

\subsection{The Normal Form of an Abstract Hardy Operator.} 

The following theorem makes the connection between a core map $B$ and its corresponding normal form parameter $b$ and shows that basic properties of the abstract Hardy operator are preserved during its transition to normal form. It is important to point out that the connection between the operators $K_B$ and $H_b$ is much closer than just sharing these basic properties, see Theorems \ref{QandR} and \ref{Phi}, Corollary \ref{bfs}, and the diagrams (\ref{cd}) for details of their relationship. 
\begin{theorem}\label{a2n} Let $(S,\Sigma,\lambda)$ and $(Y,\mu)$ be $\sigma$-finite measure spaces and suppose $B:Y\to \Sigma$ is a core map. Set $b=(\lambda\circ B)^*$, where the rearrangement is taken with respect to the measure $\mu$. Then $b$ is a normal form parameter. If $1\le p\le \infty$ and $0<q\le\infty$ then $K_B:L^p_\lambda\to L^q_\mu$ is bounded if and only if  $H_b:L^p\to L^q$ is bounded. Moreover, the operator norm of $K_B$ is $N_{p,q}(b)$. If $1<p\le q<\infty$ then $K_B:L^p_\lambda\to L^q_\mu$ is compact if and only if $H_b:L^p\to L^q$ is compact.
\end{theorem}
The proof of Theorem \ref{a2n} will be given in Section \ref{S5}. See Corollary \ref{2.3I} and Theorem \ref{acpct}.

Note especially that the formula for $b$ in terms of $B$ does not depend on the indices $p$ and $q$.

\subsection{Finding Normal Form Parameters}

The three-measure reduction given in Theorem \ref{3meas}, and the explicit formula $b=(\lambda\circ B)^*$ are not always needed to make the transition to normal form. In this section, we give two examples that illustrate the transition process. Each is followed by a theorem showing that the result of the transition holds for a large class of familiar Hardy operators. These provide simple formulas for the normal form parameter $b$ in special cases.

When $p,q\in(0,\infty)$, the normal form Hardy inequality with parameter $b$ is
\begin{equation}\label{normalform}
\bigg(\int_0^\infty  \bigg(\int_0^{b(x)}f(t)\,dt\bigg)^q\,dx\bigg)^{1/q}
\le C\bigg(\int_0^\infty  f(t)^p\,dt\bigg)^{1/p},\quad{f\in L^+}.
\end{equation}

In our first example, a weighted Hardy inequality on the half line, the transition to normal form is carried out using only substitutions and changes of variable. 
\begin{example}\label{wni} Let $1<p<\infty$ and $0<q<\infty$, let $u$ and $v$ be non-negative functions on $(0,\infty)$ and set $U(y)=\int_y^\infty u$, $w=v^{1-p'}$ and $W(s)=\int_0^sw$. We avoid technicalities by assuming that $u$ and $v$ are positive and continuous, and both $U$ and $W$ map $(0,\infty)$ onto $(0,\infty)$. As $g$ runs through $L^+$, so does $h=gw$ so the following two inequalities are equivalent:
\begin{align*}
\bigg(\int_0^\infty  \bigg(\int_0^yh(s)\,ds\bigg)^qu(y)\,dy\bigg)^{1/q}
&\le C\bigg(\int_0^\infty  h(s)^pv(s)\,ds\bigg)^{1/p},\quad h\in L^+;\\
\bigg(\int_0^\infty  \bigg(\int_0^yg(s)w(s)\,ds\bigg)^qu(y)\,dy\bigg)^{1/q}
&\le C\bigg(\int_0^\infty  g(s)^pw(s)\,ds\bigg)^{1/p},\quad g\in L^+.
\end{align*}

The changes of variable $x=U(y)$ and $t=W(s)$, and the substitution $g(s)=f(t)$, convert the second of the two to (\ref{normalform}) with parameter
$b=W\circ U^{-1}$.
\end{example}

This alternative formula for $b$ holds for Hardy inequalities with general measures.
\begin{theorem}\label{measreg} Let $1<p<\infty$ and $0<q<\infty$. Let $\lambda$ and $\mu$ be Borel measures on $\mathbb R$ such that $\lambda(-\infty,s]<\infty$ for all $s\in \mathbb R$ and  $\mu[y,\infty)<\infty$ for all $y\in\mathbb R$. Suppose $M:\mathbb R\to[0,\infty]$ satisfies $\mu(y,\infty)\le M(y)\le\mu[y,\infty)$ for all $y\in\mathbb R$ and take $M^{-1}(x)=\sup\{y\in\mathbb R:x<M(y)\}$ for all $x>0$.

With $\Lambda^-(s)=\lambda(-\infty,s)$ for $s\in\mathbb R$, $\Lambda^-\circ M^{-1}$ is the normal form parameter for the Hardy inequality
\begin{equation}\label{meas-}
\bigg(\int_{\mathbb R}  \bigg(\int_{(-\infty,y)}g\,d\lambda\bigg)^q\,d\mu(y)\bigg)^{1/q}
\le C\bigg(\int_{\mathbb R}  g^p\,d\lambda\bigg)^{1/p},\quad g\in L^+_\lambda.
\end{equation}

With $\Lambda^+(s)=\lambda(-\infty,s]$ for $s\in\mathbb R$, $\Lambda^+\circ M^{-1}$ is the normal form parameter for the Hardy inequality
\begin{equation}\label{meas+}
\bigg(\int_{\mathbb R}  \bigg(\int_{(-\infty,y]}g\,d\lambda\bigg)^q\,d\mu(y)\bigg)^{1/q}
\le C\bigg(\int_{\mathbb R}  g^p\,d\lambda\bigg)^{1/p},\quad g\in L^+_\lambda.
\end{equation}
\end{theorem}
\begin{proof} The Hardy inequalities (\ref{meas-}) and (\ref{meas+})  can be recognized as abstract Hardy inequalities by taking  $B^-(y)=(-\infty,y)$ and $B^+(y)=(-\infty,y]$, respectively. By Theorem \ref{a2n}, their normal form parameters are $(\lambda\circ B^-)^*=(\Lambda^-)^*$ and $(\lambda\circ B^+)^*=(\Lambda^+)^*$, respectively, where the rearrangement is taken with respect to the measure $\mu$. To complete the proof we show that if $\Lambda=\Lambda^-$ or $\Lambda=\Lambda^+$, then $\Lambda^*=\Lambda\circ M^{-1}$ almost everywhere on $(0,\infty)$. Since $\Lambda\circ M^{-1}$ is non-increasing on $(0,\infty)$ it suffices to show that the distribution functions $\mu_\Lambda$ and $m_{\Lambda\circ M^{-1}}$ coincide.

Let $M(y-)$ and $M(y+)$ denote the left and right limits of $M$ at $y$, respectively. Fix $y\in\mathbb R$. If $x<M(y+)$, then for some $\varepsilon>0$, $x<M(y+\varepsilon)$ so $y+\varepsilon\le M^{-1}(x)$. Thus $y<M^{-1}(x)$. Conversely, if $y<M^{-1}(x)$, then for some $\varepsilon>0$, $x<M(y+\varepsilon)$ so $x<M(y+)$. We conclude that 
\[
m\{x>0:y<M^{-1}(x)\}=m\{x>0:x<M(y+)\}=M(y+)=\mu(y,\infty).
\]
On the other hand, if $x<M(y-)$, then for all $\varepsilon>0$, $x<M(y-\varepsilon)$ so $y-\varepsilon\le M^{-1}(x)$. Thus $y\le M^{-1}(x)$. Conversely, if $y\le M^{-1}(x)$, then for all $\varepsilon>0$, $y-\varepsilon<M^{-1}(x)$ so $x< M(y-\varepsilon)$. Thus $x\le M(y-)$. We conclude that
\[
m\{x>0:y\le M^{-1}(x)\}=m\{x>0:x<M(y-)\}=M(y-)=\mu[y,\infty).
\]
Fix $t>0$. Since $\Lambda$ is non-decreasing we may choose $y_t\in[-\infty,\infty]$ so that either
\[
\{s\in\mathbb R:t<\Lambda(s)\}=(y_t,\infty)\quad\mbox{or}\quad
\{s\in\mathbb R:t<\Lambda(s)\}=[y_t,\infty).
\]
In the first case,
$
m_{\Lambda\circ M^{-1}}(t)=m\{x>0:y_t<M^{-1}(x)\}=\mu(y_t,\infty)=\mu_\Lambda(t).
$
In the second case,
$
m_{\Lambda\circ M^{-1}}(t)=m\{x>0:y_t\le M^{-1}(x)\}=\mu[y_t,\infty)=\mu_\Lambda(t).
$
This completes the proof.
\end{proof}

The next theorem follows by an analogous proof, which is omitted.
\begin{theorem}\label{measdual} Let $1<p<\infty$ and $0<q<\infty$. Let $\lambda$ and $\mu$ be Borel measures on $\mathbb R$ such that $\lambda[s,\infty)<\infty$ for all $s\in \mathbb R$ and  $\mu(-\infty,y]<\infty$ for all $y\in\mathbb R$. Suppose $M:\mathbb R\to[0,\infty]$ satisfies $\mu(-\infty,y)\le M(y)\le\mu(-\infty,y]$ for all $y\in\mathbb R$ and take $M^{-1}(x)=\inf\{y\in\mathbb R:x<M(y)\}$ for all $x>0$.

With $\Lambda^-(s)=\lambda(s,\infty)$ for $s\in\mathbb R$, $\Lambda^-\circ M^{-1}$ is the normal form parameter for the Hardy inequality
\[
\bigg(\int_{\mathbb R}  \bigg(\int_{(y,\infty)}g\,d\lambda\bigg)^q\,d\mu(y)\bigg)^{1/q}
\le C\bigg(\int_{\mathbb R}  g^p\,d\lambda\bigg)^{1/p},\quad g\in L^+_\lambda.
\]

With $\Lambda^+(s)=\lambda[s,\infty)$ for $s\in\mathbb R$, $\Lambda^+\circ M^{-1}$ is the normal form parameter for the Hardy inequality
\[
\bigg(\int_{\mathbb R}  \bigg(\int_{[y,\infty)}g\,d\lambda\bigg)^q\,d\mu(y)\bigg)^{1/q}
\le C\bigg(\int_{\mathbb R}  g^p\,d\lambda\bigg)^{1/p},\quad g\in L^+_\lambda.
\]
\end{theorem}

Taking the measures $\lambda$ and $\mu$ to be supported on $[0,\infty)$ shows that the last two theorems include the case of Hardy and dual Hardy inequalities on the half line. Choosing each of them to be a weighted Lebesgue measure gives an alternative formula for the normal form parameters of familiar weighted Hardy inequalities including Example \ref{wni}.
  
For a discrete Hardy inequality on sequence spaces the formulas for the normal form parameter given in Theorems \ref{a2n}, \ref{measreg} and \ref{measdual} may still be used. But the parameters in this special case are step functions and can be described more directly. This is done in Theorem \ref{discr}, below. In our second example, we carry out a direct transition to normal form for the dual of the discrete Hardy inequality. Once again, we make modest assumptions on the weight sequence to avoid technicalities. 
\begin{example} Let $1<p<\infty$ and $0<q<\infty$. Suppose $(u_n)$ and $(v_k)$ are positive sequences, $U_0=0$, $U_n=\sum_{k=1}^nu_k\to\infty$ as $n\to\infty$, $w_k=v_k^{1-p'}$ and $W_n=\sum_{k=n}^\infty w_k<\infty$ for each $n$. As $(g_k)$ runs through all non-negative sequences, so does $(h_k)=(g_kw_k)$. Therefore, the following inequalities are equivalent:
\begin{align}
\Big(\sum_{n=1}^\infty\Big(\sum_{k=n}^\infty h_k\Big)^qu_n\Big)^{1/q}
&\le C\Big(\sum_{k=1}^\infty h_k^pv_k\Big)^{1/p},\quad h_k\ge0;\quad\text{and}\\\label{dualdiscrete}
\Big(\sum_{n=1}^\infty\Big(\sum_{k=n}^\infty g_kw_k\Big)^qu_n\Big)^{1/q}
&\le C\Big(\sum_{k=1}^\infty g_k^pw_k\Big)^{1/p},\quad g_k\ge0.
\end{align}
Let $b$ be the step function that takes the value $W_n$ on the interval $[U_{n-1},U_n)$ for each $n$. We will show that (\ref{dualdiscrete}) and (\ref{normalform}) are equivalent. For $f\in L^+$, let $g_k=\frac1{w_k}\int_{W_{k+1}}^{W_k}f(t)\,dt$. Then by H\"older's inequality, 
\begin{equation}\label{g and f}
\sum_{k=1}^\infty g_k^pw_k\le\sum_{k=1}^\infty \int_{W_{k+1}}^{W_k}f(t)^p\,dt
=\int_0^{W_1} f(t)^p\,dt\le\int_0^\infty  f(t)^p\,dt,
\end{equation}
with equality throughout when $f$ is  constant on each $(W_{k+1},W_k)$ and zero on $(W_1, \infty)$. The left-hand side of (\ref{normalform}) becomes
\[
\Big(\sum_{n=1}^\infty\int_{U_{n-1}}^{U_n}\Big(\sum_{k=n}^\infty\int_{W_{k+1}}^{W_k}f(t)\,dt\Big)^q\,dx\Big)^{1/q}
=\Big(\sum_{n=1}^\infty\Big(\sum_{k=n}^\infty g_kw_k\Big)^qu_n\Big)^{1/q}.
\]
Thus, (\ref{dualdiscrete}) implies (\ref{normalform}). On the other hand, beginning with $(g_k)$, it is a simple matter to choose $f$ so that equality holds in (\ref{g and f}). Thus, (\ref{normalform}) implies (\ref{dualdiscrete}) as well. 
\end{example}

The alternative formula for the normal form parameter $b$, given above, holds more generally.
\begin{theorem}\label{discr} Let $1<p<\infty$ and $0<q<\infty$. Suppose $(u_n)_{n\in\mathbb Z}$ and $(w_k)_{k\in\mathbb Z}$ are two-sided sequences of non-negative numbers. If  $U_n\equiv\sum_{k\le n}u_k<\infty$ for $n\in\mathbb Z$ and $W_n\equiv\sum_{k\ge n} w_k<\infty$ for $n\in\mathbb Z$, then the normal form parameter $b$ for the Hardy inequality
\begin{equation}\label{dd}
\Big(\sum_{n\in\mathbb Z}\Big(\sum_{k=n}^\infty g_kw_k\Big)^qu_n\Big)^{1/q}
\le C\Big(\sum_{k\in\mathbb Z} g_k^pw_k\Big)^{1/p},\quad g_k\ge0,
\end{equation}
is the step function that takes the value $W_n$ on the interval $[U_{n-1},U_n)$ for each $n\in\mathbb Z$ and takes the value 0 elsewhere.

If  $U_n\equiv\sum_{k\ge n}u_k<\infty$ for $n\in\mathbb Z$ and $W_n\equiv\sum_{k\le n} w_k<\infty$ for $n\in\mathbb Z$, then the normal form parameter $b$ for the Hardy inequality
\[
\Big(\sum_{n\in\mathbb Z}\Big(\sum_{k=-\infty}^n g_kw_k\Big)^qu_n\Big)^{1/q}
\le C\Big(\sum_{k\in\mathbb Z} g_k^pw_k\Big)^{1/p},\quad g_k\ge0,
\]
is the step function that takes the value $W_n$ on the interval $[U_{n+1},U_n)$ for each $n\in\mathbb Z$ and takes the value 0 elsewhere.
\end{theorem}
\begin{proof} We prove only the first statement; the second follows from the first by reindexing $u_k$ as $u_{-k}$ and $w_k$ as $w_{-k}$. The Hardy inequality (\ref{dd})  can be recognized as an abstract Hardy inequality by setting $B(n)=\{n,n+1,n+2,\dots\}$ and taking $\lambda$ and $\mu$ to be weighted counting measures on $\mathbb Z$ with $\lambda\{k\}=w_k$ and $\mu\{n\}=u_n$. Then $\lambda\circ B(n)=W_n$. By Theorem \ref{a2n}, the normal form parameter is $b=(\lambda\circ B)^*$, where the rearrangement is taken with respect to the measure $\mu$. 

For each $t>0$, choose $n_t\in\mathbb Z$ such that $\{n\in\mathbb Z:W_n>t\}=\{n\in\mathbb Z:n\le n_t\}$. Then, for each $n\in\mathbb Z$, $n_t=n$ precisely when $W_{n+1}\le t<W_n$ so
\[
\mu_{\lambda\circ B}(t)=\mu\{n:W_n>t\}=U_{n_t}
\]
takes the value $U_n$ on $[W_{n+1},W_n)$. The generalized inverse of this function is the parameter  $b$; it takes the value $W_n$ on the interval $[U_{n-1},U_n)$.
\end{proof}

Taking $u_k=0$ and $w_k=0$ for $k\le 0$ shows that the theorem includes discrete Hardy inequalities and their duals for one-sided sequences. 

\section{Properties of Normal Form Hardy Operators}\label{S3}

We begin this section by writing classic results of Hardy and Bliss in normal form. These early results play a central role in our analysis of normal form inequalities.
\begin{proposition} If $1<p\le q<\infty$ and $\frac1s=\frac1p-\frac1q$, then 
\begin{equation}\label{Kpq}
 N_{p,q}(x^{-p'/q})=K_{p,q}\equiv\begin{cases}p^{1/p}(p')^{1/p'},&p=q;\\
\Big(\frac{\Gamma(s)}{\Gamma(s/p)\Gamma(s/q')}\Big)^{1/s},&p<q.\end{cases}
\end{equation}
\end{proposition}
\begin{proof} The best constant for the Hardy inequality,
\[
\bigg(\int_0^\infty  \bigg(\int_0^yf(t)\,dt\bigg)^qy^{-1-q/p'}\,dy\bigg)^{1/q}
\le C\bigg(\int_0^\infty  f(t)^p\,dt\bigg)^{1/p},\quad f\in L^+,
\]
was given in \cite[Theorem 327]{HLP}, when $p=q$, and in \cite{Bliss}, when $p<q$. The transition to normal form is accomplished by the change of variable $y=x^{-p'/q}$, with proper attention to the consequent change to the constant $C$.
\end{proof}
It is easy to verify directly that $K_{p,q}=K_{q',p'}$.

\subsection{Operations and Relations on the Parameter Space}
Naturally enough, Hardy operators and dual Hardy operators have been studied together from the beginnings of the theory, with a great many results formulated separately for the two. But there is no need for a ``dual normal form'' since the dual of a normal form Hardy operator is again a normal form Hardy operator. Duality remains a powerful tool in the theory but, in normal form, the distinction between ``Hardy operator'' and ``dual Hardy operator'' disappears. 

\begin{lemma}\label{dual} Let $b\in\mathcal B$ and $p,q\in[1,\infty]$. Then $b^{-1}\in\mathcal B$,
\begin{equation}\label{duality}
\int_0^\infty (H_bf)g=\int_0^\infty f(H_{b^{-1}}g), \quad f,g\in L^+,
\end{equation}
and $ N_{q',p'}(b^{-1})= N_{p,q}(b)$.
\end{lemma}
\begin{proof}
The definition of generalized inverse shows that $b^{-1}\in\mathcal B$ and (\ref{fubton}) implies (\ref{duality}). The equality of operator norms follows from the sharpness of H\"older's inequality by a standard argument.
\end{proof}

Transformations on a normal form parameter $b$ that arise from dilations of the domain and codomain lead to an innocuous non-uniqueness of the normal form. Simple changes of variable show that the norm $N_{p,q}(b)$ changes predictably when $b$ is transformed in this way. 
\begin{lemma} Let $p,q\in(0,\infty]$. If $b\in\mathcal B$, $\gamma>0$, $\delta>0$, and $a(x)=\gamma b(\delta x)$ for $x>0$, then $a\in\mathcal B$ and 
\begin{equation}\label{dilation}
 N_{p,q}(a)=\gamma^{1/p'}\delta^{-1/q} N_{p,q}(b).
\end{equation}
\end{lemma}
\begin{proof} It is clear that $a\in\mathcal B$. Replacing $t$ by $t/\gamma$ and $x$ by $\delta x$ in the normal form Hardy inequality (\ref{normalform}) gives
\[
\bigg(\int_0^\infty  \bigg(\int_0^{\gamma b(\delta x)}f(t/\gamma)\,dt\bigg)^q\,dx\bigg)^{1/q}
\le \gamma^{1/p'}\delta^{-1/q}C\bigg(\int_0^\infty  f(t/\gamma)^p\,dt\bigg)^{1/p},\ \ {f\in L^+}.
\]
Taking $C= N_{p,q}(b)$ we get $ N_{p,q}(a)\le\gamma^{1/p'}\delta^{-1/q} N_{p,q}(b)$. Since these dilations have inverse dilations, the reverse inequality also holds.
\end{proof}

The pointwise relation $a\le b$ on parameters carries over to the operator norms and an increasing sequence of parameters leads to convergence of the operator norms.  However, convergence of norms need not imply norm convergence; a stronger condition on the sequence of parameters is needed for that.
\begin{lemma}\label{cgce} Let $p,q\in(0,\infty]$ and $a,b,b_n\in\mathcal B$ for $n=1,2,\dots$. 
\begin{enumerate}[label=(\roman*)]
\item\label{cgcei} If $a\le b$ pointwise, then $N_{p,q}(a)\le N_{p,q}(b)$. 
\item\label{cgceii} If $b_n\uparrow b$ pointwise then $N_{p,q}(b_n)\to N_{p,q}(b)$.
\item\label{cgceiii} $\|H_a-H_b\|_{L^p\to L^q}\le N_{p,q}((a-b)^*)$.
\item\label{cgceiv} If $N_{p,q}((b_n-b)^*)\to0$ then $\|H_{b_n}-H_b\|_{L^p\to L^q}\to0$.
\end{enumerate}
Here $*$ denotes the rearrangement with respect to Lebesgue measure on $(0,\infty)$.
\end{lemma}
\begin{proof} The pointwise inequality, $a(x)\le b(x)$ for $x>0$, immediately implies that $H_af\le H_bf$ for all $f\in L^+$, which proves \ref{cgcei}. To prove \ref{cgceii}, suppose $b_n$ increases pointwise to $b$ and let $f\in L^+$. Then $H_{b_n}f$ increases pointwise to $H_bf$ and the monotone convergence theorem implies that $\|H_{b_n}f\|_q$ increases to $\|H_bf\|_q$. Now,
\[
N_{p,q}(b)=\sup_{0\ne f\in L^+}\frac{\|H_bf\|_q}{\|f\|_p}=\sup_{0\ne f\in L^+}\sup_n\frac{\|H_{b_n}f\|_q}{\|f\|_p}=\sup_nN_{p,q}(b_n)=\lim_{n\to\infty}N_{p,q}(b_n).
\]
If $f\in L^+$, then (\ref{HL}) shows that for each $x>0$,
\[
|(H_a-H_b)f(x)|=\Big|\int_{b(x)}^{a(x)} f\Big|\le\int_0^{|a(x)-b(x)|}f^*=(\varphi\circ|a-b|)(x),
\] 
where $\varphi(t)=\int_0^tf^*$. Since $\varphi$ is non-decreasing and continuous,  
\[
((H_a-H_b)f)^*(x)\le(\varphi\circ|a-b|)^*(x)=\varphi\circ(|a-b|^*)(x)=H_{(a-b)^*}(f^*)(x)
\]
for each $x>0$. Therefore,
\[
\|(H_a-H_b)f\|_q=\|((H_a-H_b)f)^*\|_q\le N_{p,q}((a-b)^*)\|f^*\|_p= N_{p,q}((a-b)^*)\|f\|_p.
\]
This proves \ref{cgceiii}, and \ref{cgceiv} follows directly.
\end{proof}

Part \ref{cgceiv} will be useful when we discuss compactness in Theorem \ref{cpctnormal} and is also interesting because it is not something that can be easily formulated for Hardy operators without first passing to normal form.

Any $b\in\mathcal B$ can be approximated from below by step functions in $\mathcal B$ or by smooth functions in $\mathcal B$. This is easy to prove and is easily used to simplify proofs involving normal form operators. It can also be used, with Lemma \ref{cgce}, to extend existing results for discrete Hardy inequalities or for weighted Hardy inequalities to the normal form case and hence to results for all abstract Hardy operators. 

It is not necessary to test an operator over all  functions in its domain in order to determine its operator norm. For normal form Hardy operators we may restrict our attention to well-behaved, non-increasing functions in $L^p$.
\begin{lemma}\label{decr} Let $1< p<\infty$, $0<q<\infty$, and $b\in\mathcal B$. Suppose $C\ge0$. If $\|H_bf\|_q\le C\|f\|_p$ for all non-increasing, continuous, bounded functions $f$ that are supported on $(0,n)$ for some $n$, then $N_{p,q}(b)\le C$.
\end{lemma}
\begin{proof} Suppose $\|H_bf\|_q\le C\|f\|_p$ for all functions $f$ satisfying the hypotheses.

Now fix an arbitrary $f\in L^+$. Then $\|f^*\|_p=\|f\|_p$ and, by (\ref{HL}), $H_bf(x)\le H_bf^*$. For each positive integer $n$, let $f_n(t)=n\int_t^{t+\frac1n}\min(n,f^*)\chi_{(0,n)}$. It is easy to see that for each $n$,  $f_n$ is continuous, bounded, supported in $(0,n)$ and, because it is a moving average of a non-increasing function, also non-increasing in $t$. Also, for each $t>0$, $f_n(t)$ averages its non-increasing integrand over intervals that move leftwards as $n$ increases. Since the integrand itself is non-decreasing in $n$ we see that $f_n(t)$ is non-decreasing in $n$. Since $f^*$ is right continuous, $f_n(t)\to f^*(t)$ as $n\to\infty$ for all $t>0$. 
By the monotone convergence theorem
\[
\|H_bf\|_q\le\|H_b(f^*)\|_q=\lim_{n\to\infty}\|H_b(f_n)\|_q\le C\lim_{n\to\infty}\|f_n\|_p=C\|f^*\|_p=C\|f\|_p.
\]
We conclude that $N_{p,q}(b)\le C$.
\end{proof}

\subsection{ Boundedness of Normal Form Operators} 
It is easy to calculate $ N_{p,q}(b)$ exactly in the ``endpoint'' cases, that is, for all $p,q\in(0,\infty]$  except the three cases $1<p\le q<\infty$, $1<q<p<\infty$, and $0<q<1<p<\infty$. After proving the following theorem we will focus our attention on the remaining index ranges.

\begin{theorem}\label{endpoint} Suppose $0< p\le\infty$, $0<q\le\infty$, and $b\in\mathcal B$. 
\begin{enumerate}[label=(\roman*)]
\item\label{b=0} If $b\equiv 0$ then $ N_{p,q}(b)=0$.
\item\label{p<1}  If $p<1$ then $ N_{p,q}(b)=\infty$ unless $b\equiv0$.
\item\label{1oo} If $p=1$ and $q=\infty$, then $ N_{1,\infty}(b)=1$ unless $b\equiv0$.
\item\label{1q}  If $p=1$ and $q<\infty$ then $ N_{1,q}(b)=b^{-1}(0+)^{1/q}$. 
\item\label{poo}  If $p>1$ and $q=\infty$ then $ N_{p,\infty}(b)=b(0+)^{1/p'}$.
\item\label{p1} If $p\ge1$ and $q=1$ then $ N_{p,1}(b)=\|b^{-1}\|_{p'}$.
\item\label{ooq} If $p=\infty$ and $q>0$, then $ N_{\infty,q}(b)=\|b\|_q$.
\end{enumerate}
\end{theorem}
\begin{proof} If $b\equiv 0$, $H_b$ is the zero operator so \ref{b=0} holds. 

For \ref{p<1} suppose $b\not\equiv0$, choose $x_0>0$ such that $b(x_0)>0$, and  set $f(t)=\frac1t\chi_{(0,1)}(t)$. Since $0<p<1$, $\|f\|_p<\infty$ but $H_bf(x)=\infty$ for $0<x\le x_0$ so $\|H_b\|_{L^p\to L^q}=\infty$.

If $f\in L^+$ and $x>0$ then $H_bf(x)\le \|f\|_1$ so $ N_{1,\infty}(b)\le1$. If $b\not\equiv0$, choose $x_0>0$ such that $b(x_0)>0$ and  set $f=\chi_{(0,b(x_0))}$. Then $\|H_bf\|_\infty\ge H_bf(x_0)=b(x_0)=\|f\|_1$, so $ N_{1,\infty}(b)\ge1$ and \ref{1oo} holds.

Since $\{x>0:b(x)>0\}=(0,b^{-1}(0+))$, for all $f\in L^+$, $H_bf\le\|f\|_1\chi_{(0,b^{-1}(0+))}$. Let  $q<\infty$. Then $\|H_bf\|_q\le b^{-1}(0+)^{1/q}\|f\|_1$ and $ N_{1,q}(b)\le b^{-1}(0+)^{1/q}$. On the other hand, for $\delta>0$ and $f=\chi_{(0,b(\delta))}$ we have $\|H_bf\|_q\ge \delta^{1/q}H_bf(\delta)=\delta^{1/q}\|f\|_1$. If $b\not\equiv0$, then $\|f\|_1>0$ for all $\delta<b^{-1}(0+)$, so $ N_{1,q}(b)\ge b^{-1}(0+)^{1/q}$. If $b\equiv0$ then $b^{-1}(0+)=0$ so the last inequality holds trivially. This proves \ref{1q} and Lemma \ref{dual} gives \ref{poo}.

For $f\in L^+$, $H_bf(x)\le b(x)\|f\|_\infty$ so $\|H_bf\|_{L^q}\le\|b\|_q\|f\|_\infty$, with equality in both when $f\equiv 1$. This proves \ref{ooq} and Lemma \ref{dual} gives \ref{p1}.
\end{proof}

Some parameters are too large to give rise to bounded operators. We can test for this using only characteristic functions of intervals. 
\begin{lemma}\label{niceb} Suppose $1<p<\infty$, $0<q<\infty$, $b\in \mathcal B$ and $ N_{p,q}(b)<\infty$. Then for all $x>0$, $b(x)^{1/p'}x^{1/q}\le N_{p,q}(b)$ and $b(x)<\infty$. Also, $\lim_{x\to\infty}b(x)=0$.
\end{lemma}
\begin{proof} Let $x>0$. If $0<t<b(x)$ and $f=\chi_{(0,t)}$, then $H_bf(y)=t$ when $y<x$. Therefore,
\[
tx^{1/q}=\bigg(\int_0^xH_bf(y)^q\,dy\bigg)^{1/q}
\le  N_{p,q}(b)\|f\|_p= N_{p,q}(b)t^{1/p}.
\] 
Divide by $t^{1/p}$ and let $t\to b(x)$ to get
\[
b(x)^{1/p'}x^{1/q}\le N_{p,q}(b),
\]
which also holds when $b(x)=0$. The other conclusions follow directly.
\end{proof}

The next result is the normal form version of theorems that were proved by Talenti 1969, Tomasselli 1969, Muckenhoupt 1972, Bradley 1978, Maz'ya and Rozin 1979, Kokilashvili 1979, Andersen and Heinig 1983, Maz'ya 1985, Bennett 1987 and 1991, Manakov 1992, Sinnamon 1998, Liao 2015, Ruzhansky and Verma 2019, and Li and Mao 2020. See \cite{Talenti,Toma,Muckenhoupt,Bradley,Mazya79,Kokilasvili, AH83,Mazya85,Bennett87,Bennett91, Manakov,S1D,Liao,RV19,LM20}. The large number of citations is a consequence of the result being proved first for the case $p=q$; it being proved separately for continuous, discrete, general-measure Hardy operators, and on other domains; and it being proved initially with several different constants in the upper bound before the best possible constant was found. 

For normal form Hardy operators, and therefore for all abstract Hardy operators, the characterization follows by comparing $b\in \mathcal B$ to the Hardy/Bliss parameter $x^{-p'/q}$. The constant in the upper bound is best possible because the theorem includes the Hardy-Bliss inequalities. (But see the next subsection for improved upper and lower bounds.)
 \begin{theorem}\label{Muck} Suppose $1<p\le q<\infty$. If $b\in \mathcal B$ and $A_0=\sup_{x>0} b(x)^{1/p'}x^{1/q}$, then $A_0\le  N_{p,q}(b)\le K_{p,q}A_0$.
\end{theorem}
\begin{proof} The first inequality is from Lemma \ref{niceb}. Since $b(x)\le A_0^{p'}x^{-p'/q}$ for all $x>0$, $\int_0^{b(x)}f\le\int_0^{A_0^{p'}x^{-p'/q}}f$ for all $f\in L^+$. Thus, using (\ref{dilation}) and (\ref{Kpq}), we have $ N_{p,q}(b)\le N_{p,q}(A_0^{p'}x^{-p'/q})=A_0K_{p,q}$. 
\end{proof}

The characterization of boundedness when $q<p$ also has a lengthy history. The cases $1<q<p<\infty$  and $0<q<1<p<\infty$ were considered independently and results for each case were proved separately in the continuous and discrete cases, in the case of general measures, and for Hardy operators on other domains. See Maz'ya and Rozin 1979, Sawyer 1984, Heinig 1985, Maz'ya 1985, Sinnamon 1987, Bennett 1991, Braverman and Stepanov 1992, Sinnamon and Stepanov 1996, Sinnamon 1998, Ruzhansky and Verma 2021. See \cite{Mazya79,Sawyer, Heinig85, Mazya85, thesis, Bennett91, BrSt,SS96,S1D,RV21}.

For normal form Hardy operators, and therefore for all abstract Hardy operators, the characterization follows by comparing $b\in \mathcal B$ to the Hardy/Bliss parameter $x^{1-p'}$ in the case $p=q$. The comparison is hidden within the use of Hardy's original inequality in the proof given here, but see Corollary \ref{qlpalt} for an explicit comparison. Working with normal form allows us to dodge technicalities and avoid the distinction between the cases $q>1$ and $q<1$ to give a straightforward proof. 

\begin{theorem}\label{Maz} Suppose $0<q<p<\infty$ and define $r$ by $\frac1r=\frac1q-\frac1p$. If $b\in \mathcal B$ and
\[
C_0=\bigg(\int_0^\infty  (p'b(x))^{r/p'}(qx)^{r/p}\,dx\bigg)^{1/r},
\] 
then $\frac qr C_0\le  N_{p,q}(b)\le (p')^{1/q}C_0$.
\end{theorem}
\begin{proof} If $C_0=0$ then $b\equiv0$ and the result is trivial. Both $N_{p,q}(b)$ and $C_0$ are monotone increasing in $b$ so it is sufficient to prove the result for a bounded, right continuous, compactly supported parameter $b$, for example, for $b_n=\min(n,b)\chi_{(0,n)}$ for sufficiently large $n$. Thus we may assume $0<C_0<\infty$. The identity (\ref{fubton}) shows that
\[
r\int_0^\infty (qx)^{r/p}\int_0^{b(x)}(p't)^{r/q'}\,dt\,dx
=r\int_0^\infty(p't)^{r/q'}\int_0^{b^{-1}(t)}(qx)^{r/p}\,dx\,dt.
\]
Evaluating the inner integrals, and taking $r$th roots, we get 
\begin{equation}\label{C0alt}
C_0=\bigg(\int_0^\infty (p't)^{r/q'}(qb^{-1}(t))^{r/q}\,dt\bigg)^{1/r}.
\end{equation}
Therefore, setting $f(t)=(p't)^{r/(pq')}(qb^{-1}(t))^{r/(pq)}$ gives $\|f\|_p=C_0^{r/p}$. If $t<b(x)$ then $x<b^{-1}(t)$, so
\[
H_bf(x)\ge(qx)^{r/(pq)}\int_0^{b(x)}(p't)^{r/(pq')}\,dt=\frac qr(qx)^{r/(pq)}(p'b(x))^{r/(p'q)}.
\]
Taking $q$-norms we have
\[
\frac qr C_0^{r/q}\le\|H_bf\|_q\le  N_{p,q}(b)\|f\|_p= N_{p,q}(b)C_0^{r/p}.
\]
Divide by $C_0^{r/p}$ to get the first inequality of the theorem.

Let $f\in L^+$ be non-increasing, continuous, bounded, and supported in $(0,n)$ for some $n$. Equation (\ref{fubton}) gives
\[
\|H_bf\|_q^q
=\int_0^\infty  \int_0^{b(x)}q\bigg(\int_0^tf\bigg)^{q-1}f(t)\,dt\,dx
=q\int_0^\infty  Pf(t)^{q-1}f(t)b^{-1}(t)t^{q-1}\,dt,
\]
where $Pf(t)=\frac1t\int_0^tf$. Hardy's inequality shows $\|Pf\|_p\le p'\|f\|_p$ and the monotonicity of $f$ implies $f\le Pf$.
By H\"older's inequality with indices $p/q$ and $r/q$,
\[
\|H_bf\|_q\le
\bigg(q\int_0^\infty  Pf(t)^qb^{-1}(t)t^{q-1}\,dt\bigg)^{1/q}
\le (p')^{-1/q'}C_0\|Pf\|_p\le (p')^{1/q}C_0\|f\|_p.
\]
Apply Lemma \ref{decr} to prove the second inequality of the theorem.
\end{proof}

\subsection{Comparing Normal Form Parameters}

In Theorem \ref{Muck} we compared the normal form parameter $b$ to the Hardy-Bliss parameter using a simple pointwise estimate. In this section we show that
the pointwise estimate is not the only comparison between parameters $a$ and $b$ that will imply $N_{p,q}(a)\le N_{p,q}(b)$. Many such comparisons are possible, but we have selected four to illustrate the  powerful conclusions that comparison theorems provide. The first three are applied to improve upper and lower bounds on $N_{p,q}(b)$ for fixed $p$ and $q$. The fourth is a more general form of comparison that permits a change of indices. It may be viewed as a general version of the comparison employed to prove Theorem \ref{Maz}.

Each of Theorems \ref{alt0}, \ref{alt1} and \ref{alt2}, below, begins with a comparison between parameters $a$ and $b$ that implies $N_{p,q}(a)\le N_{p,q}(b)$ and continues by first specifying $b$ to obtain an upper bound for $N_{p,q}(a)$ and then specifying $a$ to get a lower bound for $N_{p,q}(b)$. These specific parameters are closely related to the Hardy-Bliss power function parameters. Each theorem then concludes by identifying a class of parameters $b$ for which these upper and lower bounds coincide, giving the exact value of $N_{p,q}(b)$.

The first step is to introduce some truncated power function parameters.

\begin{lemma}\label{trunc} If $1<p\le q<\infty$, then for each fixed $y>0$,
\[
 N_{p,q}(x^{-p'/q}\chi_{(0,y)}(x))=K_{p,q}= N_{p,q}(\min(x^{-p'/q},y)).
\]
\end{lemma}
\begin{proof} Let $c_y(x)=x^{-p'/q}\chi_{(0,y)}(x)$. For each $\gamma>0$, $\gamma^{p'/q}c_y(\gamma x)=c_{y/\gamma}(x)$ and $c_{y/\gamma}$ increases pointwise to $x^{-p'/q}$ as $\gamma\to0$. Therefore, by (\ref{dilation}) and Lemma \ref{cgce}\ref{cgceii},
\[
 N_{p,q}(c_y(x))= N_{p,q}(\gamma^{p'/q}c_y(\gamma x))= N_{p,q}(c_{y/\gamma}(x))\to  N_{p,q}(x^{-p'/q})=K_{p,q},
\]
giving the first equation. Since  $c_y^{-1}(x)=\min(y,x^{-q/p'})$, Lemma \ref{dual} implies
\[
 N_{q',p'}(\min(x^{-q/p'},y))= N_{p,q}(c_y)=K_{p,q}=K_{q',p'}.
\]
Replacing $p$ by $q'$ and $q$ by $p'$ gives the second equation.
\end{proof}

Our first comparison theorem is for simple pointwise comparison of two normal form parameters. When one of the two is the Hardy-Bliss parameter we get a familiar upper bound for $N_{p,q}(b)$ but when one is taken to be a truncated power function we obtain a lower bound for $N_{p,q}(b)$ that does not seem to have been observed before in any class of Hardy operators.
\begin{definition}  For $1<p\le q<\infty$ and $b\in \mathcal B$, let $c_0(x)=  b(x)^{1/p'}x^{1/q}$ and
\[
 A_0=\sup_{x>0}c_0(x),\quad A_0^{(0)}=\liminf_{x\to0^+}c_0(x),\quad A_0^{(\infty)}=\liminf_{x\to\infty}c_0(x).
\]
\end{definition}

\begin{theorem}\label{alt0} Let $1<p<\infty$, $0<q<\infty$, $a$ and $b$ be normal form parameters, and $C>0$. 
\begin{enumerate}[label=(\roman*)]
\item\label{alt0i} If $a(x)\le Cb(x)$ for all $x>0$, then $N_{p,q}(a)\le C^{1/p'} N_{p,q}(b)$.
\item\label{alt0ii} If $1<p\le q<\infty$ and $b\in \mathcal B$,  then
\[
\max(A_0, K_{p,q}A_0^{(0)},K_{p,q}A_0^{(\infty)})\le  N_{p,q}(b)\le K_{p,q}A_0.
\]
\item\label{alt0iii} If $1<p\le q<\infty$, $b\in \mathcal B$, and $A_0=A_0^{(0)}$ or $A_0=A_0^{(\infty)}$, then \[
 N_{p,q}(b)= K_{p,q}A_0.
\]
\end{enumerate}
\end{theorem}
\begin{proof} If $a(x)\le Cb(x)$ for all $x>0$, then for all $f\in L^+$, $H_af\le H_{Cb}f$. Therefore $ N_{p,q}(a)\le N_{p,q}(Cb)=C^{1/p'}N_{p,q}(b)$, proving \ref{alt0i}.

If $A_0^{(0)}>0$ and $0<z<A_0^{(0)}$, choose $y>0$ so that $z^{p'}x^{-p'/q}\chi_{(0,y)}(x)\le b(x)$ for all $x>0$. By Lemma \ref{trunc} and (\ref{dilation}), $zK_{p,q}\le N_{p,q}(b)$. Letting $z\to A_0^{(0)}$, we get $K_{p,q}A_0^{(0)}\le N_{p,q}(b)$, which also holds when $A_0^{(0)}=0$.

If $A_0^{(\infty)}>0$ and $0<z<A_0^{(\infty)}$, choose $y>0$ so that $z^{p'}\min(x^{-p'/q},y)\le b(x)$ for all $x>0$. By Lemma \ref{trunc} and (\ref{dilation}), $zK_{p,q}\le N_{p,q}(b)$. Letting $z\to A_0^{(\infty)}$, we get $K_{p,q}A_0^{(\infty)}\le N_{p,q}(b)$, which also holds when $A_0^{(\infty)}=0$.

Theorem \ref{Muck} gives the remaining two bounds on $N_{p,q}(b)$ and completes \ref{alt0ii}.

Part \ref{alt0iii} follows directly from \ref{alt0ii}.
\end{proof} 

Now we introduce smaller truncated power parameters and tighten Lemma \ref{trunc}.
\begin{corollary}\label{trunc2} Suppose $1<p\le q<\infty$ and fix $y>0$. Then
\[
 N_{p,q}(\max(x^{-p'/q}-y,0))=K_{p,q}= N_{p,q}((x+y)^{-p'/q}).
\]
\end{corollary}
\begin{proof} With $b(x)=\max(x^{-p'/q}-y,0)$ it is easy to check that $1=A_0^{(0)}\le A_0\le1$. With $b(x)=(x+y)^{-p'/q}$, we have $1=A_0^{(\infty)}\le A_0\le1$. The result follows from two applications of Theorem \ref{alt0}\ref{alt0iii}.
\end{proof}
 
The next two comparison theorems take advantage of the fact that it is sufficient to test a normal-form Hardy inequality over non-negative, non-increasing functions.
\begin{definition} For $1<p\le q<\infty$ and $b\in \mathcal B$, let 
\[
c_1(t)=  t^{-1/p}\bigg(\frac{p'}p\int_{b^{-1}(t)}^\infty b^q\bigg)^{1/q}
\]
and 
\[
 A_1=\sup_{t>0}c_1(t),\quad A_1^{(0)}=\liminf_{t\to0^+}c_1(t),\quad A_1^{(\infty)}=\liminf_{t\to\infty}c_1(t).
\]
\end{definition}
\begin{theorem}\label{alt1} Let $1<p<\infty$, $0<q<\infty$, $a$ and $b$ be normal form parameters, and $C>0$. 
\begin{enumerate}[label=(\roman*)]
\item\label{alt1i} If, for all $t>0$, 
\[
\int_{a^{-1}(t)}^\infty a^q\le C\int_{b^{-1}(t)}^\infty b^q
\]
then $N_{p,q}(a)\le C^{1/q} N_{p,q}(b)$.
\item\label{alt1ii} If $1<p\le q<\infty$ and $b\in \mathcal B$,  then 
\[
\max((p/p')^{1/q}A_1, K_{p,q}A_1^{(0)},K_{p,q}A_1^{(\infty)})\le N_{p,q}(b)\le K_{p,q}A_1.
\]
\item\label{alt1iii} If $1<p\le q<\infty$, $b\in \mathcal B$, and $A_1^{(0)}=A_1$ or $A_1^{(\infty)}=A_1$, then 
\[
N_{p,q}(b)=K_{p,q}A_1.
\]
\end{enumerate}
\end{theorem}
\begin{proof} Let $f\in L^+$ be non-increasing, continuous, bounded, and supported in $(0,n)$ for some $n$. Then  $\big(\frac1y\int_0^yf\big)^q$ is decreasing, continuously differentiable and tends to zero at infinity so it can be written as $\int_y^\infty\sigma$ for a non-negative $\sigma$. Applying (\ref{fubton}) to the inequality
\[
\int_0^\infty\sigma(t)\int_{a^{-1}(t)}^\infty a(x)^q\,dx\,dt
\le C\int_0^\infty\sigma(t)\int_{b^{-1}(t)}^\infty b(x)^q\,dx\,dt
\]
yields
\[
\int_0^\infty\int_{a(x)}^\infty\sigma(t)\,dt a(x)^q\,dx\le C\int_0^\infty\int_{b(x)}^\infty\sigma(t)\,dyb(x)^q\,dx,
\]
which simplifies to 
\[
\int_0^\infty\bigg(\int_0^{a(x)}f\bigg)^q\,dx \le C\int_0^\infty\bigg(\int_0^{b(x)}f\bigg)^q\,dx.
\]
Now Lemma \ref{decr} implies $ N_{p,q}(a)\le C^{1/q} N_{p,q}(b)$. This proves \ref{alt1i}.

For \ref{alt1ii}, let $t>0$ and set $f=\chi_{(0,t)}$. We are free to suppose that $b$ is right continuous, so by (\ref{fub}), if $x\ge b^{-1}(t)$ then $t\ge b(x)$, so
\[
\bigg(\int_{b^{-1}(t)}^\infty b(x)^q\,dx\bigg)^{1/q}
=\bigg(\int_{b^{-1}(t)}^\infty\bigg(\int_0^{b(x)}f\bigg)^q\,dx\bigg)^{1/q}
\le N_{p,q}(b)\|f\|_p=N_{p,q}(b) t^{1/p}.
\]
Thus, $(p/p')^{1/q}A_1\le N_{p,q}(b)$.

For the upper bound, observe that
\[
\int_{b^{-1}(t)}^\infty b^q\le (p/p')A_1^qt^{q/p}=A_1^q\int_{t^{-q/p'}}^\infty (x^{-p'/q})^q\,dx
\]
so \ref{alt1i} implies $N_{p,q}(b)\le A_1N_{p,q}(x^{-p'/q})=K_{p,q}A_1$.

To establish the remaining two lower bounds we will apply \ref{alt1i} twice. If $A_1^{(\infty)}>0$ and $0<z<A_1^{(\infty)}$, choose $y>0$ so that for all $t>y^{-p'/q}$,
\begin{equation}\label{z1}
z\le t^{-1/p}\Bigg(\frac{p'}p\int_{b^{-1}(t)}^\infty b^q\Bigg)^{1/q}.
\end{equation}
With $a(x)=x^{-p'/q}\chi_{(0,y)}(x)$ we have $a^{-1}(t)=\min(y, t^{-q/p'})$ and
\[
\int_{a^{-1}(t)}^\infty a^q=\frac{t^{q/p}-y^{1-p'}}{p'-1}\chi_{(y^{-p'/q},\infty)}(t)
\le\frac p{p'}t^{q/p}\le z^{-q}\int_{b^{-1}(t)}^\infty b^q.
\]
Therefore \ref{alt1i} and Lemma \ref{trunc} imply, after letting $z\to A_1^{(\infty)}$, that
\[
K_{p,q}A_1^{(\infty)}=N_{p,q}(a)A_1^{(\infty)}\le N_{p,q}(b)
\]
which also holds when $A_1^{(\infty)}=0$.

If $0<z<A_1^{(0)}$, choose $y>0$ so that for all $t\le y$, inequality (\ref{z1}) holds. Take $a(x)=(x+y^{-q/p'})^{-p'/q}$ and get  $a^{-1}(t)=(t^{-q/p'}-y^{-q/p'})\chi_{(0,y)}(t)$. If $t<y$, then
\[
\int_{a^{-1}(t)}^\infty a^q=\int_{t^{-q/p'}-y^{-q/p'}}^\infty (x+y^{-q/p'})^{-p'}\,dx=\frac  p{p'}t^{q/p}\le z^{-q}\int_{b^{-1}(t)}^\infty b^q.
\]
If $t>y$, then
\[
\int_{a^{-1}(t)}^\infty a^q=
\int_0^\infty (x+y^{-q/p'})^{-p'}\,dx=\frac{y^{q/p}}{p'-1}\le z^{-q}\int_{b^{-1}(y)}^\infty b^q\le z^{-q}\int_{b^{-1}(t)}^\infty b^q.
\]
Therefore \ref{alt1i} and Corollary \ref{trunc2} imply, after letting $z\to A_1^{(0)}$, that
\[
K_{p,q}A_1^{(0)}=N_{p,q}(a)A_1^{(0)}\le N_{p,q}(b),
\]
which also holds when $A_1^{(0)}=0$.
This completes the proof of \ref{alt1ii}. 

Part \ref{alt1iii} follows from \ref{alt1ii}.
\end{proof}

To show that this comparison can achieve results that the simple pointwise comparison of Theorem \ref{alt0} cannot we offer an example.

\begin{example} In a 1921 letter to Hardy, Landau showed that if $1<p<\infty$, then $C=p'$ is the best constant in the discrete Hardy inequality,
\[
\Big(\sum_{n=1}^\infty\Big(\frac1n\sum_{k=1}^nf_k\Big)^p\Big)^{1/p}
\le C\Big(\sum_{k=1}^\infty f_k^p\Big)^{1/p},\quad f_k\ge0.
\]
Prior to the letter the inequality had been proven, but only with constants larger than $p'$. See \cite{KMP06} or \cite[Appendix]{KMP07}.

Theorem \ref{alt0} will give a proof but only with $C>p'$. However, Theorem \ref{alt1}\ref{alt1iii} easily gives the correct constant: Let $U_n=\sum_{k=n}^\infty k^{-p}$. According to Theorem \ref{discr}, the normal form parameter $b$ takes the value $b(x)=n$ when $U_{n+1}< x<U_n$ and $b(x)=0$ when $x>U_1$. So $b^{-1}(t)=U_n$ when $n-1<t<n$ and we have
\[
\int_{b^{-1}(t)}^\infty b^p=\sum_{k=1}^{n-1}\int_{U_{k+1}}^{U_k}k^p
=\sum_{k=1}^{n-1}k^{-p}k^p=n-1.
\]
Thus, 
\[
\Big(\frac{n-1}n\Big)^{1/p}\le t^{-1/p}\Bigg(\int_{b^{-1}(t)}^\infty b^p\Bigg)^{1/p}\le 1.
\]
It follows that $(p'/p)^{1/p}=A_1^{(\infty)}\le A_1\le(p'/p)^{1/p}$ so Theorem \ref{alt1}\ref{alt1iii} implies $N_{p,p}(b)=(p'/p)^{1/p}K_{p,p}=p'$. This gives another proof of Landau's result. Notice that no separate lower bounds are required.
\end{example}

Our next comparison theorem improves the upper and lower bounds arising from direct pointwise comparison.
\begin{definition}  For $1<p\le q<\infty$ and $b\in \mathcal B$, let 
\[
c_2(t)=t^{-1/q'}\bigg(\frac{p'}{q'}\int_0^t x^{p'-1}b(x)\,dx\bigg)^{1/p'},
\]
and 
\[
 A_2=\sup_{t>0}c_2(t),\quad A_2^{(0)}=
\liminf_{t\to0^+}c_2(t),\quad A_2^{(\infty)}=\liminf_{t\to\infty}c_2(t).
\]
\end{definition}

\begin{theorem}\label{alt2} Let $1<p<\infty$, $1<q<\infty$, $a$ and $b$ be normal form parameters, and $C>0$. 
\begin{enumerate}[label=(\roman*)]
\item\label{alt2i} If, for all $t>0$, 
\[
\int_0^tx^{p'-1}a(x)\,dx\le C\int_0^tx^{p'-1}b(x)\,dx
\]
then $N_{p,q}(a)\le C^{1/p'}N_{p,q}(b)$.
\item\label{alt2ii} If $1<p\le q<\infty$,  then 
\[
\max((q')^{1/p'}A_2,K_{p,q}A_2^{(0)},K_{p,q}A_2^{(\infty)})\le N_{p,q}(b)\le K_{p,q}A_2.
\]
\item\label{alt2iii} If $1<p\le q<\infty$ and $A_2^{(0)}=A_2$ or $A_2^{(\infty)}=A_2$ then $N_{p,q}(b)=K_{p,q}A_2$.

\end{enumerate}
\end{theorem}
\begin{proof} Let $f\in L^+$ be non-increasing, continuous, bounded, and supported in $(0,n)$ for some $n$. Then
\[
(H_{b^{-1}}f)(t)^{p'}=\Big(\int_0^{b^{-1}(t)}f\Big)^{p'}
=\int_0^{b^{-1}(t)}p' \Big(\int_0^xf\Big)^{p'-1}f(x)\,dx.
\]
Now let $F(x)=p'\big(\frac1x\int_0^xf\big)^{p'-1}f(x)$ and apply (\ref{fubton}) to get
\begin{equation}\label{fF}
\int_0^\infty (H_{b^{-1}}f)(t)^{p'}\,dt
=\int_0^\infty F(x)x^{p'-1}b(x)\,dx.
\end{equation}
The same equation is true with the parameter $b$ replaced by $a$. Since $f$ is non-increasing, so is $F$, and a standard argument shows that the hypothesis of \ref{alt2i} implies
\[
\int_0^\infty F(x)x^{p'-1}a(x)\,dx\le C\int_0^\infty F(x)x^{p'-1}b(x)\,dx.
\]
That is,
\[
\int_0^\infty (H_{a^{-1}}f)(t)^{p'}\,dt\le C\int_0^\infty (H_{b^{-1}}f)(t)^{p'}\,dt.
\]
By Lemma \ref{decr}, we have 
\[
N_{p,q}(a)=N_{q',p'}(a^{-1})\le C^{1/p'}N_{q',p'}(b^{-1})=C^{1/p'}N_{p,q}(b).
\]
This proves \ref{alt2i}.

For \ref{alt2ii}, fix $t>0$ and let $f=\chi_{(0,t)}$. Then $\|f\|_{q'}=t^{1/q'}$. The function $F$, introduced in the proof of \ref{alt2i}, becomes $p'\chi_{(0,t)}$ and we use (\ref{fF}) to get
\[
p'\int_0^tx^{p'-1}b(x)\,dx
=\int_0^\infty (H_{b^{-1}}f)(t)^{p'}\,dt\le (N_{q'/p'}(b^{-1})t^{1/q'})^{p'}=N_{p,q}(b)^{p'}t^{p'/q'}.
\]
Thus, $(q')^{1/p'}A_2\le N_{p,q}(b)$.

For the upper bound, observe that for all $t>0$
\[
\int_0^tx^{p'-1}b(x)\,dx\le\frac{q'}{p'}A_2^{p'}t^{p'/q'}=A_2^{p'}\int_0^tx^{p'-1}x^{-p'/q}\,dx
\]
so \ref{alt2i} implies $N_{p,q}(b)\le A_2N_{p,q}(x^{-p'/q})=K_{p,q}A_2$.

To prove the lower bound we apply \ref{alt2i} twice. If $A_2^{(0)}>0$ and $0<z<A_2^{(0)}$, choose $y$ so that for all $t\le y$, 
\begin{equation}\label{z2}
z<t^{-1/q'}\bigg(\frac{p'}{q'}\int_0^tx^{p'-1}b(x)\,dx\bigg)^{1/p'}.
\end{equation}
Let $a(x)=x^{-p'/q}\chi_{(0,y)}$. If $t\le y$, then
\[
\int_0^tx^{p'-1}a(x)\,dx=\int_0^tx^{(p'/q')-1}\,dx=\frac{q'}{p'}t^{p'/q'}
\le z^{-p'}\int_0^tx^{p'-1}b(x)\,dx.
\]
If $t>y$, then
\[
\int_0^tx^{p'-1}a(x)\,dx=\int_0^yx^{(p'/q')-1}\,dx
\le z^{-p'}\int_0^yx^{p'-1}b(x)\,dx\le z^{-p'}\int_0^tx^{p'-1}b(x)\,dx.
\]
Now Lemma \ref{trunc} and \ref{alt2i} imply, after letting $z\to A_2^{(0)}$, that 
\[
K_{p,q}A_2^{(0)}=A_2^{(0)}N_{p,q}(a)\le N_{p,q}(b),
\]
which also holds when $A_2^{(0)}=0$.

Next, if $A_2^{(\infty)}>0$ and $0<z<A_2^{(\infty)}$, choose $y$ so that for all $t\ge y$, 
inequality (\ref{z2}) holds. Let $a(x)=(x+y)^{-p'/q}$. If $t\ge y$, then
\[
\int_0^tx^{p'-1}a(x)\,dx\le\int_0^tx^{p'-1}x^{-p'/q}\,dx=\frac{q'}{p'}t^{p'/q'}\le z^{-p'}\int_0^tx^{p'-1}b(x)\,dx.
\]
If $t\le y$, then, since $b$ is non-increasing, so are its averages. Using (\ref{z2}) with $t$ replaced by $y$, we have
\[
z^{p'}y^{-p'/q}\le q'z^{p'}y^{-p'/q}< \frac{p'}{y^{p'}}\int_0^yx^{p'-1}b(x)\,dx\le\frac{p'}{t^{p'}}\int_0^tx^{p'-1}b(x)\,dx.
\]
Therefore,
\[
\int_0^tx^{p'-1}a(x)\,dx
\le y^{-p'/q}\int_0^tx^{p'-1}\,dx
=\frac{t^{p'}}{p'}y^{-p'/q}
\le z^{-p'}\int_0^tx^{p'-1}b(x)\,dx.
\]
Now Corollary \ref{trunc2} and part \ref{alt2i} imply, after letting $z\to A_2^{(\infty)}$, that 
\[
K_{p,q}A_2^{(\infty)}=A_2^{(\infty)}N_{p,q}(a)\le N_{p,q}(b),
\]
which also holds when $A_2^{(\infty)}=0$. This completes the proof of \ref{alt2ii}.

Part \ref{alt2iii} follows directly from part \ref{alt2ii}.
\end{proof}

Next we show that both the upper and lower bounds for $N_{p,q}(b)$ given by Theorem \ref{alt2} are at least as good as those given by Theorem \ref{alt0}. As a consequence, the class of parameters for which the best constant is given exactly in Theorem \ref{alt2}\ref{alt2iii} contains all those whose best is constant given exactly in Theorem \ref{alt0}\ref{alt0iii}.

\begin{lemma}\label{tighter} For $1<p\le q<\infty$ and $b\in \mathcal B$, then 
\[
A_2\le A_0\le (q')^{1/p'}A_2,\quad A_0^{(0)}\le A_2^{(0)}\quad\mbox{and }\quad A_0^{(\infty)}\le A_2^{(\infty)}.
\] 
Consequently,
\begin{align*}
\max(A_0, K_{p,q}A_0^{(0)},K_{p,q}A_0^{(\infty)})&\le\max((q')^{1/p'}A_2,K_{p,q}A_2^{(0)},K_{p,q}A_2^{(\infty)})\\&\le N_{p,q}(b)
\le K_{p,q}A_2\le K_{p,q}A_0.
\end{align*}
\end{lemma}
\begin{proof} Substituting the estimate $b(x)\le A_0^{p'}x^{-p'/q}$ in the definition of $c_2$ gives $c_2(t)\le A_0$ for all $t>0$. Thus $A_2\le A_0$. For any $y>0$, using the estimate $b(x)\ge b(y)\chi_{(0,y)}(x)$, valid for any non-increasing function, in the definition of $c_2$ gives $c_2(t)\ge (q')^{-1/p'}b(y)^{1/p'}t^{-1/q'}\min(t,y)$, which takes its largest value at $t=y$. Thus $(q')^{1/p'}A_2\ge b(y)^{1/p'}y^{1/q}$ for all $y$ and we conclude that $(q')^{1/p'}A_2\ge A_0$. 

If $A_0^{(0)}>0$ and $0<z<A_0^{(0)}$, choose $y>0$ so that if $x<y$, then $z<b(x)^{1/p'}x^{1/q}$. If $t<y$, then 
\[
c_2(t)\ge t^{-1/q'}\bigg(\frac{p'}{q'}\int_0^t x^{p'-1}z^{p'}x^{-p'/q}\,dx\bigg)^{1/p'}=z.
\]
It follows that $A_0^{(0)}\le A_2^{(0)}$, which also holds when $A_0^{(0)}=0$.

If $A_0^{(\infty)}>0$ and $0<z<A_0^{(\infty)}$, choose $y$ so that if $x>y$, then $z<b(x)^{1/p'}x^{1/q}$. Fix $\varepsilon\in(0,1)$. If $t>y/\varepsilon>y$, then 
\[
c_2(t)\ge t^{-1/q'}\bigg(\frac{p'}{q'}\int_y^t x^{p'-1}z^{p'}x^{-p'/q}\,dx\bigg)^{1/p'}
\ge z(1-\varepsilon^{p'/q'})^{1/p'}.
\]
Letting $\varepsilon\to0$ and $z\to A_0^{(\infty)}$ shows that $A_0^{(\infty)}\le A_2^{(\infty)}$, which also holds when $A_0^{(\infty)}=0$.
\end{proof}

Theorem \ref{q<p} provides a comparison between normal form parameters that permits a change in indices. If $q>1$ then duality may be used to avoid the appearance of $a^{-1}$ and $b^{-1}$ and give a statement in terms of the original parameters $a$ and $b$. This reformulation is left to the reader.

\begin{theorem}\label{q<p} Suppose $\alpha\in(0,1]$, $1<p_0\le p<\infty$, $1<q_0<\infty$, $0<q\le q_0$, 
\[
q/q_0\le\alpha<q\quad\mbox{and}\quad  (q-\alpha)(p_0-1)<(p-\alpha)(q_0-1).
\]
If $a$ and $b$ are normal form parameters, then
\[
\big(q^{-1/q}N_{p,q}(a)\big)^{\frac{pq}{q-\alpha}}\le
\bigg(\int_0^\infty\bigg(\frac{(t^{q-1}a^{-1}(t))^{\frac{p}{q-\alpha}}}{(t^{q_0-1}b^{-1}(t))^{\frac {p_0}{q_0-1}}}\bigg)^\kappa\,dt\bigg)^{1/\kappa}\big(q_0^{-1/q_0}N_{p_0,q_0}(b)\big)^{\frac{p_0q_0}{q_0-1}},
\]
where $\frac1\kappa=\frac{p-\alpha}{q-\alpha}-\frac{p_0-1}{q_0-1}$.
\end{theorem}
\begin{proof} Let $f\in L^+$ be non-increasing, continuous, bounded, and supported in $(0,n)$ for some $n$. Then we have $f(t)\le\frac1t\int_0^tf$ and $\frac1t\int_0^tf\le\Big(\frac1t\int_0^tf^{p/p_0}\Big)^{p_0/p}$, so
\begin{align*}
\|H_af\|_{q}^q
&=\int_0^\infty q\bigg(\int_0^tf\bigg)^{q-1}f(t)a^{-1}(t)\,dt\\
&\le q\int_0^\infty\bigg(\int_0^tf\bigg)^{q-\alpha}f(t)^{\alpha}t^{\alpha-1}a^{-1}(t)\,dt\\
&\le q\int_0^\infty\bigg(\frac1t\int_0^tf^{p/p_0}\bigg)^{p_0(q-\alpha)/p}f(t)^{\alpha}t^{q-1}a^{-1}(t)\,dt\\
&=q\int_0^\infty\bigg[\bigg(\int_0^tf^{p/p_0}\bigg)^{q_0-1}f(t)^{p/p_0}b^{-1}(t)\bigg]^{\frac{p_0(q-\alpha)}{p(q_0-1)}}f(t)^{\frac{q_0\alpha-q}{q_0-1}}G(t)^{\frac{q-\alpha}{p}}\,dt,
\end{align*}
where 
\[
G(t)=\frac{(t^{q-1}a^{-1}(t))^{\frac{p}{q-\alpha}}}{(t^{q_0-1}b^{-1}(t))^{\frac {p_0}{q_0-1}}}.
\]
We apply H\"older's inequality with the indices $\frac{p(q_0-1)}{p_0(q-\alpha)}$, $\frac{p(q_0-1)}{q_0\alpha-q}$ and $\frac{p \kappa}{q-\alpha}$ to see that 
\[
\|H_af\|_{q}^{q}\le q\bigg(\int_0^\infty\bigg(\int_0^tf^{p/p_0}\bigg)^{q_0-1}f(t)^{p/p_0}b^{-1}(t)\,dt\bigg)^{\frac{p_0(q-\alpha)}{p(q_0-1)}}\|f\|_{p}^{\frac{q_0\alpha-q}{q_0-1}}\|G\|_{\kappa}^{\frac{q-\alpha}{p}}.
\]
(If $q_0\alpha-q$ is zero, apply H\"older's inequality with just two indices, to get the same result.)
By (\ref{fubton}),
\begin{align*}
\bigg(q_0\int_0^\infty\bigg(\int_0^tf^{p/p_0}\bigg)^{q_0-1}f(t)^{p/p_0}b^{-1}(t)\,dt\bigg)^{1/q_0}
&=\|H_b(f^{p/p_0})\|_{q_0}\\
&\le N_{p_0,q_0}(b)\|f\|_p^{p/p_0}.
\end{align*}
Combining these estimates, and simplifying, yields
\[
\|H_af\|_q\le q^{1/q}\big(q_0^{-1/q_0} N_{p_0,q_0}(b)\big)^{\frac{p_0q_0(q-\alpha)}{pq(q_0-1)}}
\|f\|_p\|G\|_{\kappa}^{\frac{q-\alpha}{pq}}.
\]
Now Lemma \ref{decr} implies
\[
N_{p,q}(a)\le q^{1/q}\big(q_0^{-1/q_0} N_{p_0,q_0}(b)\big)^{\frac{p_0q_0(q-\alpha)}{pq(q_0-1)}}
\|G\|_{\kappa}^{\frac{q-\alpha}{pq}},
\]
which can be rearranged to give the conclusion of the theorem.
\end{proof}

The special case $p_0=q_0=p$, and $b(x)=x^{1-p'}$ (and $a$ renamed to $b$) is stated as a corollary. It improves the upper bound in Theorem \ref{Maz}.

\begin{corollary}\label{qlpalt}
Suppose $\alpha\in(0,1]$, $1<p<\infty$,  $0<q< p$, and $q/p\le\alpha<q$. Define $r$ by $1/r=1/q-1/p$. If $b$ is a normal form parameter and 
\[
 C_0=\bigg(\int_0^\infty (p't)^{r/q'}(qb^{-1}(t))^{r/q}\,dt\bigg)^{1/r},
\]
then $ N_{p,q}(b)\le (p')^{(1-\alpha)/q}C_0$. 
\end{corollary}

By (\ref{C0alt}), this $C_0$ agrees with the one in Theorem \ref{Maz}. Take $\alpha=1$ when $q>1$ and let $\alpha\to q$ when $q\le1$ to get 
\[
 N_{p,q}(b)\le\max(1, (p')^{-1/q'})C_0.
\]
This is better than $N_{p,q}(b)\le(p')^{1/q}C_0$, given in Theorem \ref{Maz}. In particular, it recovers the correct value of $N_{p,1}(b)$, from Theorem \ref{endpoint}\ref{p1}.

\subsection{Compactness of Normal Form Operators} There is a simple necessary and sufficient condition for a normal form Hardy operator to be compact, in terms of the parameter $b$. In view of Theorem \ref{a2n} the same condition is equivalent to the compactness of every abstract Hardy operator that has $b$ as its normal form parameter.

We begin with a construction that ensures non-compactness of $H_b$ and certain other closely related operators. 
\begin{lemma}\label{cpctlm} Suppose $1<p\le q<\infty$. Let $(Y,\mu)$ be a $\sigma$-finite measure space and suppose $\Lambda:Y\to[0,\infty]$ satisfies $\Lambda^*(x)<\infty$ for all $x>0$, and $\Lambda^*(x)\to0$ as $x\to\infty$.
With $b=\Lambda^*$, define 
\[
f_x=\frac{\chi_{(0,b(x))}}{\|\chi_{(0,b(x))}\|_p}\quad\text{and}\quad
F_x(y)=\int_0^{\Lambda(y)}f_x.
\]
If $\limsup_{x\to\infty}b(x)^{1/p'}x^{1/q}>\varepsilon>0$, then there exists a strictly increasing sequence $(x_n)$ such that $\|F_{x_n}-F_{x_m}\|_{L^q_\mu}>\varepsilon$ whenever $m>n$.

If $\limsup_{x\to0+}b(x)^{1/p'}x^{1/q}>\varepsilon>0$, then there exists a strictly decreasing sequence $(x_n)$ such that $\|F_{x_n}-F_{x_m}\|_{L^q_\mu}>\varepsilon$ whenever $m>n$.
\end{lemma}
\begin{proof} Suppose $\limsup_{x\to\infty}b(x)^{1/p'}x^{1/q}>\varepsilon$ and choose $x_1$ with $b(x_1)^{1/p'}x_1^{1/q}>\varepsilon$. Since $b(x)\to0$ as $x\to\infty$, we can choose $x_2,x_3,\dots$, recursively, such that $x_1<x_2<x_3,\dots$, and for each $n$,
\[
b(x_{n+1})^{1/p'}x_{n+1}^{1/q}>\varepsilon\quad\mbox{and}\quad
(b(x_n)^{1/p'}-b(x_{n+1})^{1/p'})x_n^{1/q}>\varepsilon.
\]
Note that for all $x>0$, $f_x=b(x)^{-1/p}\chi_{(0,b(x))}$ and $F_x(y)=\min(\Lambda(y),b(x))b(x)^{-1/p}$. Since $\Lambda$ and $b$ are equimeasurable, $\mu_\Lambda(b(x)-)=m_b(b(x)-)$, and we have
\begin{equation}\label{distn}
\mu\{y\in Y:\Lambda(y)\ge b(x)\}=m\{t>0:b(t)\ge b(x)\}\ge m\{t>0:t<x\}=x
\end{equation}
for all $x>0$. If $m>n$, then $x_n<x_{n+1}\le x_m$ so $b(x_n)\ge b(x_{n+1})\ge b(x_m)$. Thus,
\begin{align*}
\|F_{x_n}-F_{x_m}\|_{L^q_\mu}&\ge\Bigg(\int_{\{y\in Y:\Lambda(y)\ge b(x_n)\}}(b(x_n)^{1/p'}-b(x_m)^{1/p'})^q\,d\mu(y)\Bigg)^{1/q}\\&\ge(b(x_n)^{1/p'}-b(x_m)^{1/p'})x_n^{1/q}\\&\ge (b(x_n)^{1/p'}-b(x_{n+1})^{1/p'})x_n^{1/q}>\varepsilon.
\end{align*}

Now suppose $\limsup_{x\to0+}b(x)^{1/p'}x^{1/q}>\varepsilon$ and choose $x_1$ with $b(x_1)^{1/p'}x_1^{1/q}>\varepsilon$. Among $\{x>0:b(x)^{1/p'}x^{1/q}>\varepsilon\}$, $b(x)\to\infty$ as $x\to 0$. So we can choose $z_1, x_2, z_2, x_3, \dots$, recursively, such that for each $n$, $x_1>z_1>x_2>z_2>x_3>\dots$, and 
\begin{align*}
&b(x_n)^{1/p'}(x_n-z_n)^{1/q}>\varepsilon,\\
&b(x_{n+1})^{1/p'}x_{n+1}^{1/q}>\varepsilon,\quad\mbox{and}\\
&(b(x_n)^{1/p'}-b(z_n)b(x_{n+1})^{-1/p})(x_n-z_n)^{1/q}>\varepsilon.
\end{align*}
(Recall that $b(z_n)<\infty$ by hypothesis.)

Since $\mu_\Lambda(b(z))=m_b(b(z))$,
\[
\mu\{y\in Y:\Lambda(y)> b(z)\}=m\{t>0:b(t)> b(z)\}\le m\{t>0:t<z\}=z
\]
for all $z>0$. Combining this with (\ref{distn}), we see that if $0<z<x$, then
\[
\mu\{y\in Y:b(z)\ge\Lambda(y)\ge b(x)\}\ge x-z.
\]
If $m>n$, then $x_n>z_n>x_{n+1}\ge x_m$ so $b(x_n)\le b(z_n)\le b(x_{n+1})\le b(x_m)$. Thus,
\begin{align*}
&\|F_{x_n}-F_{x_m}\|_{L^q_\mu}\\&\ge\Bigg(\int_{\{y\in Y:b(z_n)\ge\Lambda(y)\ge b(x_n)\}}(b(x_n)^{1/p'}-\Lambda(y)b(x_m)^{-1/p})^q\,d\mu(y)\Bigg)^{1/q}\\&\ge(b(x_n)^{1/p'}-b(z_n)b(x_m)^{-1/p})(x_n-z_n)^{1/q}\\&\ge(b(x_n)^{1/p'}-b(z_n)b(x_{n+1})^{-1/p})(x_n-z_n)^{1/q}>\varepsilon.
\end{align*}
This completes the proof.

\end{proof}

\begin{theorem}\label{cpctnormal} Suppose $1<p\le q<\infty$ and let $b$ be a normal form parameter. Let $b_n=2^{-n}\sum_{k=1}^{n2^n}\chi_{(0,b^{-1}(k2^{-n}))}$, for $n=1,2,\dots$. The following are equivalent for $H_b: L^p\to L^q$:
\begin{enumerate}[label=(\roman*)]
\item\label{cpct4}  $H_b$ is compact;
\item\label{cpct3}  $H_b$ is a norm limit of finite-rank operators;
\item\label{cpct2} $H_{b_n}\to H_b$ in operator norm;
\item\label{cpct1} $\lim_{x\to\infty}b(x)^{1/p'}x^{1/q}=0$ and $\lim_{x\to0+}b(x)^{1/p'}x^{1/q}=0$.
\end{enumerate}
\end{theorem}
\begin{proof} We assume that $b$ is right continuous; this has no effect on $b^{-1}$, $b_n$, $H_b$, or the limits in \ref{cpct1}. Using right continuity and (\ref{fub}), it is straightforward to show that $b_1,b_2\dots$ is an increasing sequence of normal form parameters that converges pointwise almost everywhere to $b$. On the interval $(b^{-1}(n),\infty)$, $b-b_n\le 2^{-n}$. For each $f$, $H_{b_n}f$ is in the span of $\{\chi_{(0,b^{-1}(k2^{-n}))}:k=1,2,\dots,n2^n\}$ so $H_{b_n}$ is a finite-rank operator. Thus \ref{cpct2} implies \ref{cpct3}. That \ref{cpct3} implies \ref{cpct4} is well known. It remains to show \ref{cpct1} implies \ref{cpct2} and \ref{cpct4} implies \ref{cpct1}.

Assume \ref{cpct1} holds and fix $\varepsilon>0$. Set $\gamma=(\varepsilon/K_{p,q})^{p'}$. Choose $x_0$ so small and $x_\infty$ so large that $b(x)\le\gamma x^{-p'/q}$ for all $x\in(0,x_0]\cup[x_\infty,\infty)$. Evidently, $b(x)<\infty$ for all $x>0$ and it follows that $b^{-1}(n)\to0$ as $n\to\infty$. Choose $n$ so large that $b^{-1}(n)<x_0$ and $2^{-n}<\gamma x_\infty^{-p'/q}$. If $x_0<x<x_\infty$, then
\[
b(x)-b_n(x)\le 2^{-n}<\gamma x_\infty^{-p'/q}\le\gamma x^{-p'/q}
\]
and if $x\in(0,x_0]\cup[x_\infty,\infty)$, then
\[
b(x)-b_n(x)\le b(x)<\gamma x^{-p'/q}.
\]
Since $(b-b_n)(x)$ is dominated by the non-increasing function $\gamma x^{-p'/q}$, so is its rearrangement, $(b-b_n)^*$. By Lemma \ref{cgce}\ref{cgceiii} and Theorem \ref{alt0}\ref{alt0i}, 
\[
\|H_b-H_{b_n}\|_{L^p\to L^q}\le  N_{p,q}((b-b_n)^*)\le\gamma^{1/p'}K_{p,q}=\varepsilon.
\]
This proves \ref{cpct2}.

Assume that \ref{cpct1} fails. It follows that either
$\limsup_{x\to\infty}b(x)^{1/p'}x^{1/q}>0$ or $\limsup_{x\to0+}b(x)^{1/p'}x^{1/q}>0$. If $ N_{p,q}(b)=\infty$, then $H_b$ is not bounded and so cannot be compact. Thus \ref{cpct4} fails. If $ N_{p,q}(b)<\infty$ then Lemma \ref{niceb} ensures that $b(x)<\infty$ for all $x>0$ and $b(x)\to 0$ as $x\to\infty$ so we may apply Lemma \ref{cpctlm}, with $(Y,\mu)=((0,\infty),m)$ and $\Lambda=b$.
The result is a sequence of unit vectors $f_{x_n}\in L^p$ for which $H_bf_{x_n}$ has no convergent subsequence. So \ref{cpct4} fails. We conclude that \ref{cpct4} implies \ref{cpct1}.
\end{proof}

\section{Ordered Cores}\label{S4}

In this section we introduce and explore the notion of an \emph{ordered core} of a measure space. No order is assumed on the elements of the space. Instead, a totally ordered collection of measurable sets is chosen to carry the order. The main result is that an order-preserving map from this core collection into the measurable sets of another measure space induces a positive, admissible map of measurable functions on the two spaces.

\begin{definition} Let $(P,\mathcal P, \rho)$ be a $\sigma$-finite measure space. An \emph{ordered core} of $(P,\mathcal P, \rho)$ is a totally ordered subset $\mathcal A$ of $\mathcal P$, containing $\emptyset$, that consists of sets of finite $\rho$-measure. The ordered core $\mathcal A$ is \emph{$\sigma$-bounded} if there exists a countable subset $\mathcal A_0$ of $\mathcal A$ such that $\cup\mathcal A=\cup\mathcal A_0$.
\end{definition}

\begin{definition} Let  $\mathcal A$ be an ordered core of $(P,\mathcal P, \rho)$ and let $(T,\mathcal T,\tau)$ be a $\sigma$-finite measure space. A map $r :\mathcal A\to\mathcal T$ is \emph{order-preserving} if $r(A_1)\subseteq r(A_2)$ whenever $A_1\subseteq A_2$. An order-preserving map $r :\mathcal A\to\mathcal T$ is \emph{bounded} if there exists a constant $c>0$ such that for all $A,B\in\mathcal A$ with $A\subseteq B$, 
\begin{equation}\label{cbound}
\tau(r (B)\setminus r (A))\le c\rho(B\setminus A).
\end{equation}
\end{definition}

We explore these notions in a series of three lemmas. The first investigates the ring of sets generated by an ordered core. The second takes a bounded, order-preserving map and uses it to define a family of premeasures on this ring by pulling back the integrals of positive functions. The third extends these premeasures to measures on a certain measure space generated by the ordered core.

With this preparation we are able, in Theorem \ref{map}, to use Radon-Nikodym differentiation of this family of measures to define the desired positive, admissible map.

Recall that a \emph{ring} of subsets (of a given set) is a non-empty collection that is closed under finite intersection, finite union and relative complementation. A \emph{semiring} is one that is closed under finite intersection and in which each relative complement can be expressed as a finite disjoint union of sets from the collection. See \cite[p354, 357]{RF}.

\begin{lemma}\label{L1}  Let $(P,\mathcal P, \rho)$ be a $\sigma$-finite measure space and let $\mathcal A$ be an ordered core of $(P,\mathcal P, \rho)$. Set $\mathcal A^+=\{B\setminus A:A,B\in\mathcal A\}$ and let $\mathcal A^{++}$ be the collection of all finite disjoint unions of elements of $A^+$. If $A,B,A_1, B_1,A_2, B_2\in\mathcal A$, then
\begin{enumerate}[label=(\roman*)]

\item\label{1}$A\cap B, A\cup B\in\mathcal A$;

\item\label{2}$(B_1\setminus A_1)\cap(B_2\setminus A_2)=(B_1\cap B_2)\setminus (A_1\cup A_2)\in\mathcal A^+$;

\item\label{3}$(B_2\setminus A_2)\setminus (B_1\setminus A_1)=[B_2\setminus (A_2\cup B_1)]\cup[(B_2\cap B_1\cap A_1)\setminus A_2]$, a disjoint union of elements of $\mathcal A^+$; and

\item\label{4} if $\emptyset\ne B_1\setminus A_1\subseteq B_2\setminus A_2$, then $A_2\subseteq A_1\subseteq B_1\subseteq B_2$. 
\end{enumerate}
In particular, $\mathcal A^+$ is a semiring and $\mathcal A^{++}$ is a ring. 
\end{lemma}
\begin{proof} Since $\mathcal A$ is totally ordered, \ref{1} is evident, and \ref{2} and \ref{3} are just calculations. To prove \ref{4}, let $x\in B_1\setminus A_1$. Then $x\in B_1$, $x\notin A_1$, $x\in B_2$, and $x\notin A_2$. Since $\mathcal A$ is totally ordered, these imply $A_1\subseteq B_1$, $A_2\subseteq B_1$ and $A_1\subseteq B_2$. Now $A_2=A_2\setminus (B_2\setminus A_2)\subseteq B_1\setminus (B_1\setminus A_1)=A_1\subseteq B_1=A_1\cup(B_1\setminus A_1)\subseteq B_2\cup(B_2\setminus A_2)=B_2$. Finally, since $\emptyset\in\mathcal A$, \ref{2} and \ref{3} show that $\mathcal A^+$ is a semiring and the proof of \cite[Proposition 13, p354]{RF} shows that $\mathcal A^{++}$ is a ring.
\end{proof}

Next, we introduce another measure space and an order-preserving map on our ordered core.

\begin{lemma}\label{L2}  Let $(P,\mathcal P, \rho)$ and $(T,\mathcal T,\tau)$ be $\sigma$-finite measure spaces, let $\mathcal A$ be an ordered core of $(P,\mathcal P, \rho)$ and let $r :\mathcal A\to\mathcal T$ be a bounded, order-preserving map. For each $f\in L_\rho^+$ set 
\begin{equation}\label{rhof}
\rho_f(B\setminus A)=\int_{r (B)\setminus r (A)}f\,d\tau.
\end{equation}
If $A,B,A_k, B_k\in\mathcal A$ for $k=1,2,\dots$ and $f\in L_\tau^+$, then 
\begin{enumerate}[label=(\roman*)]
\item\label{4+} $r (A\cap B)=r (A)\cap r (B)$ and $r (A\cup B)=r (A)\cup r (B)$;

\item\label{fred} if $A_1\subseteq A_2\subseteq A_3$, then $\rho_f(A_3\setminus A_1)=\rho_f(A_3\setminus A_2)+\rho_f(A_2\setminus A_1)$;

\item\label{6} if $B\setminus A=\emptyset$ then $\rho_f(B\setminus A)=0$;

\item\label{7} if $B_1\setminus A_1=B_2\setminus A_2$ then $\rho_f(B_1/A_1)=\rho_f(B_2/A_2)$, that is, $\rho_f$ is well defined;

\item\label{8} if $B_1\setminus A_1$ and $B_2\setminus A_2$ are disjoint and $B\setminus A=(B_1\setminus A_1)\cup(B_2\setminus A_2)$ then  $\rho_f(B/A)=\rho_f(B_1/A_1)+\rho_f(B_2/A_2)$;

\item\label{9} for each $n$, if $B\setminus A=\cup_{k=1}^n(B_k\setminus A_k)$, a disjoint union,  then  $\rho_f(B/A)=\sum_{k=1}^n\rho_f(B_k/A_k)$;

\item\label{10} for each $n$, if $B\setminus A\subseteq\cup_{k=1}^n(B_k\setminus A_k)$ then  $\rho_f(B/A)\le\sum_{k=1}^n\rho_f(B_k/A_k)$; and

\item\label{11} if $B\setminus A\subseteq\cup_{k=1}^\infty(B_k\setminus A_k)$ then  $\rho_f(B/A)\le\sum_{k=1}^\infty\rho_f(B_k/A_k)$.
\end{enumerate}
Thus $\rho_f$ is a premeasure on the semiring $\mathcal A^+=\{B\setminus A:A,B,\in\mathcal A\}$. See \cite[p353]{RF}.
\end{lemma}
\begin{proof} Since $\mathcal A$ is totally ordered and $r $ preserves order, \ref{4+} and \ref{fred} are immediate and, for \ref{6}, $B\setminus A=\emptyset$ implies $r (B)\setminus r (A)=\emptyset$ and hence $\rho_f(B\setminus A)=0$. If $B_1\setminus A_1=\emptyset$ then \ref{7} follows from \ref{6}; if not, Lemma \ref{L1}\ref{4} implies $A_1=A_2$ and $B_1=B_2$ so \ref{7} holds trivially.

If either $B_1\setminus A_1$ or $B_2\setminus A_2$ is empty, then \ref{8} follows from \ref{6} and \ref{7}. Otherwise, suppose, without loss of generality, that $B_2\subseteq B_1$. Since $\emptyset\ne B_1\setminus A_1\subseteq B\setminus A$ and $\emptyset\ne B_2\setminus A_2\subseteq B\setminus A$, Lemma \ref{L1}\ref{4} shows that $A_1\subseteq B_1\subseteq B$ and $A\subseteq A_2\subseteq B_2$. But $B_1\setminus A_1$ and $B_2\setminus A_2$ are disjoint so Lemma \ref{L1}\ref{3} gives
\[
 B_2\setminus A_2=(B_2\setminus A_2)\setminus (B_1\setminus A_1)=(B_2\cap A_1)\setminus A_2\subseteq A_1\setminus A_2.
\]
By Lemma \ref{L1}\ref{4}, $B_2\subseteq A_1$ and we conclude that $A\subseteq A_2\subseteq B_2\subseteq A_1\subseteq B_1\subseteq B$. But $B\setminus A=(B_1\setminus A_1)\cup(B_2\setminus A_2)$, so $A=A_2$, $B_2=A_1$ and $B_1=B$, reducing the assertion to part \ref{fred}.

We prove \ref{9} by induction on $n$. The case $n=1$ is \ref{7}. Suppose the assertion holds for $n-1$ and
$B\setminus A=\cup_{k=1}^n(B_k\setminus A_k)$ is a disjoint union. By reordering, we may assume $B_k\subseteq B_n$ for all $k$. Then $B_n\setminus A_n\subseteq B\setminus A=\cup_{k=1}^n (B_k\setminus A_k)\subseteq B_n$ so 
\[
B\setminus A=((B\cap A_n)\setminus A)\cup(B_n\setminus A_n)
\quad\mbox{and}\quad (B\cap A_n)\setminus A=\cup_{k=1}^{n-1}(B_k\setminus A_k),
\]
both disjoint unions. Now \ref{8}, followed by the inductive hypothesis completes the proof.

We also prove \ref{10} by induction on $n$.  If $B\setminus A=\emptyset$, the case $n=1$ follows from \ref{6} and if not, it follows from Lemma \ref{L1}\ref{4} and two applications of \ref{fred}. Suppose the assertion holds for $n-1$ and
$B\setminus A\subseteq\cup_{k=1}^n(B_k\setminus A_k)$. Again, we reorder so that $B_k\subseteq B_n$ for all $k$. Then $B\setminus A\subseteq B_n$ so $B\setminus (A\cup A_n)\subseteq B_n\setminus A_n$. But $B\setminus A=((B\cap A_n)\setminus A)\cup(B\setminus (A\cup A_n))$, a disjoint union, and 
\[
(B\cap A_n)\setminus A=(B\setminus A)\cap A_n=\cup_{k=1}^n(B_k\setminus A_k)\cap A_n\subseteq\cup_{k=1}^{n-1}(B_k\setminus A_k)
\]
so \ref{8}, the case $n=1$, and the inductive hypothesis, show that
\[
\rho_f(B\setminus A)=\rho_f((B\cap A_n)\setminus A)+\rho_f(B\setminus (A\cup A_n))
\le\sum_{k=1}^{n-1}\rho_f(B_k\setminus A_k)+\rho_f(B_n\setminus A_n).
\]

Every $f\in L^+_\tau$ may be expressed as an increasing pointwise limit of non-negative, integrable functions. Applying the monotone convergence theorem in (\ref{rhof}) we see that it suffices to prove \ref{11} in the case $f\in L^+_\tau\cap L^1_\tau$. Fix such an $f$.

By hypothesis, $r$ is bounded. Let $c$ be the constant in (\ref{cbound}). Fix $\varepsilon>0$. Choose a $\delta>0$ such that for any $\tau$-measurable $T_0$, if $\tau(T_0)<\delta$ then $\int_{T_0}f\,d\tau<\varepsilon$. That this can be done is a well-known property of integrable functions. 

Now suppose $B\setminus A\subseteq\cup_{k=1}^\infty(B_k\setminus A_k)$ and let $S_n=(B\setminus A)\setminus \big(\cup_{k=1}^nB_k\setminus A_k\big)$. Then $S_1,S_2,\dots$ is a decreasing sequence of subsets of $B\setminus A$ with empty intersection. Since $\rho(B\setminus A)\le\rho(B)<\infty$ we may choose $n$ so that $\rho(S_n)<\delta/c$.

The set $S_n$ is in the ring $\mathcal A^{++}$ so $S_n$ can be expressed as a finite disjoint union of elements of $\mathcal A^+$, say $S_n=\cup_{j=1}^m (B'_j\setminus A'_j)$. But $r $ preserves order, so the set $\cup_{n=1}^m(r (B'_j)\setminus r (A'_j))$ is also a disjoint union. The boundedness of $r $ gives 
\[
\tau(\cup_{n=1}^mr (B'_j)\setminus r (A'_j))=\sum_{j=1}^m\tau(r (B'_j)\setminus r (A'_j))\le c\sum_{j=1}^m\rho(B'_j\setminus A'_j)=c\rho(S_n)<\delta
\]
whence
\[
\sum_{j=1}^m\rho_f(B'_j\setminus A'_j)=\sum_{j=1}^m\int_{r (B'_j)\setminus r (A'_j)}f\,d\tau=\int_{\cup_{n=1}^mr (B'_j)\setminus r (A'_j)}f\,d\tau<\varepsilon.
\]
But $B\setminus A\subseteq \big(\cup_{k=1}^nB_k\setminus A_k\big)\cup\big(\cup_{j=1}^m B'_j\setminus A'_j\big)$ so, by \ref{10},
\[
\rho_f(B\setminus A)\le\Big(\sum_{k=1}^n\rho_f(B_k\setminus A_k)\Big)+\Big(\sum_{j=1}^m \rho_f(B'_j\setminus A'_j)\Big)\le
\sum_{k=1}^\infty\rho_f(B_k\setminus A_k)+\varepsilon.
\]
Since $\varepsilon$ was arbitrary, the result follows.

Parts \ref{6}, \ref{9} and \ref{11} show that $\rho_f$ is a premeasure on the semiring $\mathcal A^+$.
\end{proof}

Now we have a family of premeasures $\rho_f$ indexed by $f\in L^+_\tau$. It is standard measure theory to extend them to the $\sigma$-algebra generated by the semiring $\mathcal A^+$. However, that extension may lead to a measure space that is too large for our purposes. So instead, we introduce a measure space over the smallest subset of $P$ that contains all the sets in the ordered core. On this more natural measure space, we extend our premeasures to produce a family of measures and then investigate two important subfamilies, those coming from an $f$ in $L^1_\tau\cap L^+_\tau$ and those coming from an $f$ in $L^\infty_\tau\cap L^+_\tau$.

\begin{lemma}\label{L3}  Let $(P,\mathcal P, \rho)$ and $(T,\mathcal T,\tau)$ be a $\sigma$-finite measure spaces, let $\mathcal A$ be a $\sigma$-bounded ordered core of $(P,\mathcal P, \rho)$ and let $r :\mathcal A\to\mathcal T$ be a bounded, order-preserving map. Set $\hat P\equiv\cup\mathcal A$, take $\hat{\mathcal P}$ to be the $\sigma$-ring generated by $\mathcal A$, and let $\hat \rho$ be the restriction of $\rho$ to $\hat{\mathcal P}$. For each $f\in L_\tau^+$ define $\rho_f$ by (\ref{rhof}), let $\rho_f^*$ be the outer measure induced by $\rho_f$ on subsets of $\hat P$ and let $\bar \rho_f$ be the restriction of $\rho_f^*$ to $\hat{\mathcal P}$. 

Let $f\in L^+_\tau$. Then,
\begin{enumerate}[label=(\roman*)]

\item\label{12} $(\hat P,\hat{\mathcal P},\hat\rho)$ is a $\sigma$-finite measure space and $\bar\rho_f$ is a measure on $(\hat P,\hat{\mathcal P})$ that extends $\rho_f$;

\item\label{13} if $f\in L^1_\tau+L^\infty_\tau$ then $\bar\rho_f$ is the unique $\sigma$-finite extension of $\rho_f$ to $\hat{\mathcal P}$ and $\bar\rho_f\ll\hat\rho$;

\item\label{14} if $f\in L^1_\tau$, then $\bar\rho_f(\hat P)\le\|f\|_{L^1_\tau}$; and

\item\label{15} if $f\in L^\infty_\tau$, then for all $E\in\hat{\mathcal P}$,  $\bar\rho_f(E)\le c\|f\|_{L^\infty_\tau}\hat\rho(E)$.
\end{enumerate}
\end{lemma}
\begin{proof} Since $\mathcal A$ is $\sigma$-bounded, $\hat P$ is a $\sigma$-finite, $\rho$-measurable set that lies in $\hat{\mathcal P}$. Therefore, the $\sigma$-ring $\hat{\mathcal P}$ is a $\sigma$-algebra of $\rho$-measurable subsets of $\hat P$. It follows that $(\hat P,\hat{\mathcal P},\hat\rho)$ is a $\sigma$-finite measure space. (While $\hat{\mathcal P}$ is a sub-$\sigma$-ring of $\mathcal P$, it will not be a sub-$\sigma$-algebra of $\mathcal P$ unless $P=\hat P$.)

Fix $f\in L^+_\tau$. Since $\rho_f$ is a premeasure on the semiring $\mathcal A^+$ and $\mathcal A^+\subseteq \hat{\mathcal P}$, the Carath\'eodory measure induced by $\rho_f$ on the $\rho_f^*$-measurable subsets of $\hat P$ is an extension of $\rho_f$. See \cite[p356]{RF}. In particular, each set in $\mathcal A^+$ is $\rho_f^*$-measurable and it follows that all sets in $\hat{\mathcal P}$ are also $\rho_f^*$-measurable. Thus $\bar\rho_f$ is the restriction to $\hat{\mathcal P}$ of this Carath\'eodory measure and is therefore a measure on $\hat{\mathcal P}$ that extends $\rho_f$. This proves \ref{12}.

For \ref{13}, suppose $f\in L^1_\tau+L^\infty_\tau$. Since $f$ is non-negative, an easy argument shows that $f=f_1+f_\infty$ for some non-negative functions $f_1\in L^1_\tau$ and $f_\infty\in L^\infty_\tau$. If $A\in\mathcal A$, then, in view of (\ref{cbound}), we have
\[
\bar\rho_f(A)=\rho_f(A)=\int_{r(A)\setminus r(\emptyset)}f_1\,d\tau+\int_{r(A)\setminus r(\emptyset)}f_\infty\,d\tau\le\|f\|_{L^1_\tau}+\|f\|_{L^\infty_\tau}c\rho(A)<\infty.
\]
But $\mathcal A$ is $\sigma$-bounded so $\hat P$ can be expressed as a countable union of sets of finite $\bar\rho_f$-measure. Thus $\bar\rho_f$ is $\sigma$-finite. Uniqueness now follows from \cite[Corollary 14, p357]{RF}. 

To complete \ref{13} we show that $\bar\rho_f$ is absolutely continuous with respect to $\hat \rho$. Suppose $E\in\hat{\mathcal P}$ and $\hat\rho(E)=0$. Fix $\varepsilon>0$. Choose $\delta>0$ so small that $\|f_\infty\|_{L^\infty_\tau}\delta<\varepsilon/2$ and if $\tau(T_0)<\delta$ then $\int_{T_0}f_1\,d\tau<\varepsilon/2$.

Consider the restriction of $\hat\rho$ to the semiring $\mathcal A^+$. It is a premeasure, and because $\mathcal A$ is a $\sigma$-bounded ordered core it is also $\sigma$-finite. Therefore, the outer measure based on this restriction coincides with the measure $\hat\rho$ on $\hat{\mathcal P}$. Moreover, because the restriction is a premeasure on a semiring, in the definition of this outer measure it suffices to consider covers by disjoint sets. Therefore we can choose a disjoint sequence $B_1\setminus A_1,B_2\setminus A_2,\dots$, with $A_k,B_k\in\mathcal A$, such that $E\subseteq\cup_{k=1}^\infty B_k\setminus A_k$ and $\sum_{k=1}^\infty\hat\rho(B_k\setminus A_k)<\delta/c$. Then  $r (B_1)\setminus r (A_1),r (B_2)\setminus r (A_2),\dots$ is also a disjoint sequence. Set $T_0=\cup_{k=1}^\infty r (B_k)\setminus r (A_k)$. Because $\rho$ and $\hat\rho$ coincide on $\mathcal A^+$, 
\[
\tau(T_0)
\le\sum_{k=1}^\infty\tau(r (B_k\setminus A_k))
\le c\sum_{k=1}^\infty\rho(B_k\setminus A_k)
< \delta.
\]
It follows that,
\[
\int_{T_0}f_1\,d\tau<\varepsilon/2\quad\mbox{and}\quad
\int_{T_0}f_\infty\,d\tau<\varepsilon/2.
\]
Therefore,
\[
\bar\rho_f(E)=\rho_f^*(E)\le\sum_{k=1}^\infty\rho_f(B_k\setminus A_k)=\sum_{k=1}^\infty\int_{r (B_k)\setminus r (A_k)}f\,d\tau
=\int_{T_0}f\,d\tau<\varepsilon.
\]
Letting $\varepsilon\to0$ we see that $\bar\rho_f(E)=0$. This shows $\bar\rho_f\ll\hat\rho$.

For \ref{14}, since $\mathcal A$ is $\sigma$-bounded and $\mathcal A^+$ is a semiring, we may choose a disjoint sequence $B_1\setminus A_1,B_2\setminus A_2,\dots$, with $A_k,B_k\in\mathcal A$, such that $\hat P=\cup_{k=1}^\infty B_k\setminus A_k$. Then $r (B_1)\setminus r (A_1),r (B_2)\setminus r (A_2),\dots$ is also a disjoint sequence and 
\[
\bar\rho_f(\hat P)\le\sum_{k=1}^\infty\rho_f(B_k\setminus A_k)=\sum_{k=1}^\infty\int_{r (B_k)\setminus r (A_k)}f\,d\tau
\le\int_T f\,d\tau=\|f\|_{L^1_\tau}.
\]

A similar estimate proves \ref{15}: Fix $E\in\hat{\mathcal P}$ and let $B_1\setminus A_1,B_2\setminus A_2,\dots$ be any disjoint sequence with $A_k,B_k\in\mathcal A$, such that $E\subseteq\cup_{k=1}^\infty B_k\setminus A_k$. Because $\rho$ and $\hat\rho$ coincide on $\mathcal A^+$ and $r$ is bounded,
\[
\bar\rho_f(E)\le\sum_{k=1}^\infty\rho_f(B_k\setminus A_k)
=\sum_{k=1}^\infty\int_{r (B_k)\setminus r (A_k)}f\,d\tau
\le c\|f\|_{L^\infty_\tau}\sum_{k=1}^\infty\hat\rho(B_k\setminus A_k).
\]
Taking the infimum over all such sequences and using the fact, established in the proof of \ref{13}, that $\hat\rho$ coincides with its outer measure on $\hat{\mathcal P}$, we get
\[
\bar\rho_f(E)\le c\|f\|_{L^\infty_\tau}\hat\rho(E).
\]
\end{proof}

The next theorem is the main result of the section. In it we use Radon-Nikodym differentiation of the measures $\bar\rho_f$ to define an operator that takes $\tau$-measurable functions to $\rho$-measurable functions in a way that preserves the integrals over core sets. 

\begin{theorem}\label{map} Let $(P,\mathcal P, \rho)$ and $(T,\mathcal T,\tau)$ be a $\sigma$-finite measure spaces, let $\mathcal A$ be a $\sigma$-bounded ordered core of $(P,\mathcal P, \rho)$ and let $r :\mathcal A\to\mathcal T$ be a bounded, order-preserving map. Then there is a positive linear map $R:L^1_\tau+L^\infty_\tau\to L^1_\rho+ L^\infty_\rho$ such that $\|R f\|_{L^1_\rho}\le\|f\|_{L^1_\tau}$ for all $f\in L^1_\tau$; $\|R f\|_{L^\infty_\rho}\le c\|f\|_{L^\infty_\tau}$ for all $f\in L^\infty_\tau$; and 
\begin{equation}\label{respect}
\int_{B\setminus A}R f\,d\rho=\int_{r (B)\setminus r (A)}f\,d\tau,
\end{equation}
for all $f\in L^1_\tau+L^\infty_\tau$ and all $A,B\in\mathcal A$. 
\end{theorem}
\begin{proof}  In the proof we will use the definition and notation introduced in Lemmas \ref{L2} and \ref{L3}. 

First we define $R$ on non-negative functions only. Fix a non-negative function $f\in L^1_\tau+L^\infty_\tau$. By Lemma \ref{L3}\ref{13}, both $\hat\rho$ and $\bar\rho_f$ are $\sigma$-finite and $\bar\rho_f\ll\hat\rho$. Define $R f$ to be the Radon-Nikodym derivative of $\bar\rho_f$ with respect to $\hat\rho$ on $\hat P$ and $R f$ to be zero on $P\setminus \hat P$. Note that $R f$ restricted to $\hat P$ is $\hat{\mathcal P}$-measurable and $R f$ is $\mathcal P$-measurable. It is clear that $Rf$ is non-negative. For all $A,B\in\mathcal A$, $B\setminus A\in\hat{\mathcal P}$ so we have
\[
\int_{B\setminus A}R f\,d\rho=\int_{B\setminus A}R f\,d\hat\rho=\bar\rho_f(B\setminus A)=\rho_f(B\setminus A)=\int_{r (B)\setminus r (A)}f\,d\tau
\]
so (\ref{respect}) holds. By Lemma \ref{L3}\ref{14} if $f\in L^1_\tau$, then
\[
\|R f\|_{L^1_\rho}=\int_PR f\,d\rho=\int_{\hat P}R f\,d\hat\rho=\bar\rho_f(\hat P)\le\|f\|_{L^1_\tau}.
\]
If $f\in L^\infty_\tau$ and $0<\alpha<\|R f\|_{L^\infty_\rho}$, then the set $\{x\in \hat P:R f(x)>\alpha\}$ has positive $\hat\rho$-measure. By the $\sigma$-finiteness of $\hat\rho$, it has a subset $E_\alpha$ of positive, finite $\hat\rho$-measure. By Lemma \ref{L3}\ref{15},
\[
\alpha\hat\rho(E_\alpha)\le\int_{E_\alpha}R f\,d\hat\rho= \bar\rho_f(E_\alpha)\le c\|f\|_{L^\infty_\tau}\hat\rho(E_\alpha).
\]
Dividing by $\hat\rho(E_\alpha)$ and letting $\alpha\to\|R f\|_{L^\infty_\rho}$ shows that $\|R f\|_{L^\infty_\rho}\le c\|f\|_{L^\infty_\tau}$.

Next we show that $R$ is additive and homogeneous for non-negative scalars. If $f$ and $g$ are non-negative and $\tau$-integrable, and $\alpha$, $\beta$ are non-negative reals, then $\alpha f+\beta g$ is non-negative and $\tau$-integrable. It follows directly from the definition (\ref{rhof}) that $\rho_{\alpha f+\beta g} = \alpha\rho_f+\beta\rho_g$. The measures $\bar\rho_{\alpha f+\beta g}$ and $\alpha\bar\rho_f+\beta\bar\rho_g$ agree on $\mathcal A^+$ and both extend $\rho_{\alpha f+\beta g}$ to $\hat{\mathcal P}$. Uniqueness of the extension shows that the two measures coincide. In the presence of $\sigma$-finiteness, the Radon-Nikodym derivative is unique so we have $R (\alpha f+\beta g)=\alpha R f+\beta R g$ $\hat\rho$-almost everywhere in $\hat P$. Both sides vanish on $P\setminus\hat P$ so equality holds $\rho$-almost everywhere in $P$.

Finally, we extend $R$ by complex-linearity to a map from  $L^1_\tau+L^\infty_\tau$ to $L^1_\rho+ L^\infty_\rho$. The key property, (\ref{respect}), is clearly preserved and a routine argument shows that the boundedness properties of $R$ are also preserved. 
\end{proof}

\section{Abstract Hardy Inequalities}\label{S5}

Let $(S,\Sigma,\lambda)$ and $(Y,\mu)$ be $\sigma$-finite measure spaces and suppose $B:Y\to \Sigma$ is a core map. From Definition \ref{coremap}, this means
\begin{enumerate}[label=\rm({\Roman*})]
\item\label{to2} (Total orderedness) the range of $B$ is a totally ordered subset of $\Sigma$;
\item\label{me2} (Measurability) for $E\in\Sigma$, $y\mapsto\lambda(E\cap B(y))$ is $\mu$-measurable;
\item\label{se2} ($\sigma$-boundedness) $\cup_{y\in Y}B(y)=\cup_{y\in Y_0}B(y)$ for some countable $Y_0\subseteq Y$;
\item\label{fi2} (Finiteness) for all $y\in Y$, $\lambda(B(y))<\infty$. 
\end{enumerate}
The abstract Hardy operator associated with the map $B$ is
\[
K_B(g)=\int_{B(y)}g\,d\lambda, \quad\mbox{for } g\in L^+_\lambda.
\]
For $p,q\in(0,\infty]$ and $C\ge0$, the associated abstract Hardy inequality is
\begin{equation}\label{2measure}
\bigg(\int_Y\bigg(\int_{B(y)}g\,d\lambda\bigg)^q\,d\mu(y)\bigg)^{1/q}
\le C\bigg(\int_Sg^p\,d\lambda\bigg)^{1/p}, \quad\mbox{for all } g\in L^+_\lambda,
\end{equation}
where $C\in[0,\infty]$.

Abstract Hardy inequalities will be our main objects of study, but they are not the most general Hardy inequalities that arise in practice. To explain this seeming lack of generality we introduce the following three-measure Hardy inequality and show in the next theorem that if $p>1$, then it either fails for simple reasons or it is equivalent to a certain abstract Hardy inequality. In most situations it is simpler to reduce a three-measure Hardy inequality directly to an abstract Hardy inequality, before analyzing it, than it is to apply the following theorem. However, it is important to know that the reduction is possible.

Fix $p\in(1,\infty)$ and $q\in(0,\infty)$. Let $(Z,\mu)$ be a $\sigma$-finite measure space and let $\nu$ and $\omega$ be $\sigma$-finite measures on a measurable space $(S,\Sigma)$. Apply the Lebesgue decomposition theorem and the Radon-Nikodym theorem to get a non-negative $\Sigma$-measurable function $u$ and $\sigma$-finite measures $\nu^\perp$ and $\lambda$ satisfying
\[
\nu^\perp\perp\omega,\quad\omega=\nu^\perp +u\nu\quad\mbox{and}\quad\lambda=u^{p'}\nu.
\]
Set $Y=\{z\in Z:\lambda(B(z))<\infty\}$.

\begin{theorem}\label{3meas} Fix a finite constant $C\ge0$. Let $p$, $q$, $(Z,\mu)$, $(S,\Sigma)$, $\nu$, $\omega$, $u$, $\nu^\perp$, $\lambda$ and $Y$ be as in the preceding paragraph. Suppose $B:Z\to\Sigma$ satisfies:
\begin{enumerate}[label=(\roman*)]
\item\label{3a} the range of $B$ is a totally ordered subset of $\Sigma$;
\item\label{3b} for each $E\in\Sigma$, $z\mapsto\omega(E\cap B(z))$ is a $\mu$-measurable function on $Z$;
\item\label{3c} there is a countable $Y_0\subseteq Y$ such that $\cup_{y\in Y}B(y)=\cup_{y\in Y_0}B(y)$;
\item\label{3d} there is a countable $Z_0\subseteq Z\setminus Y$ such that $\cup_{z\in Z\setminus Y}A(y)=\cup_{z\in Z_0}A(y)$, where $A(z)=\{z'\in Z:\lambda(B(z)\setminus B(z'))=0\}$.
\end{enumerate}
Then (\ref{2measure}) is an abstract Hardy inequality. Also,
\begin{equation}\label{3measure}
\bigg(\int_Z\bigg(\int_{B(z)}h\,d\omega\bigg)^q\,d\mu(z)\bigg)^{1/q}
\le C\bigg(\int_Sh^p\,d\nu\bigg)^{1/p},\quad\mbox{for all } h\in L^+_\nu,
\end{equation}
holds if and only if $\nu^\perp(B(z))=0$ $\mu$-almost everywhere on $Z$, $\mu(Z\setminus Y)=0$ and (\ref{2measure}) holds.
\end{theorem}

\begin{proof} To begin we show that both (\ref{2measure}) and (\ref{3measure}) are meaningful as written. Property \ref{3b} implies that the map $z\mapsto \int_{B(z)}h\,d\omega$ is $\mu$-measurable for each $h\in L^+_\nu$. Thus, the left-hand side of (\ref{3measure}) is defined. 

The orthogonality of $\nu$ and $\nu^\perp$ provides an $S_0\in\Sigma$ such that $\nu(S_0)=0$ and $\nu^\perp(S\setminus S_0)=0$. Let $E\in\Sigma$ and set $h=u^{p'-1}\chi_E\chi_{S\setminus S_0}\in L^+_\nu$. For each $z\in Z$,
\[
\int_{B(z)}h\,d\omega=\int_{B(z)\cap S_0}h\,d\nu^{\perp} + \int_{B(z)\setminus S_0}hu\,d\nu=\int_{B(z)\setminus S_0}\chi_E\,d\lambda=\lambda(E\cap B(z)).
\]
Therefore, $z\mapsto \lambda(E\cap B(z))$ is a $\mu$-measurable function. Taking $E=S$, we see that $Y=\{z\in Z:\lambda(B(z))<\infty\}$ is a $\mu$-measurable set. Let $\mu|_Y$ denote the restriction of $\mu$ to the measurable subsets of $Y$. Then the restriction to $Y$ of the map $z\mapsto \lambda(E\cap B(z))$ is a $\mu|_Y$-measurable function on $Y$. Thus, for any $g\in L^+_\lambda$, $y\mapsto \int_{B(y)}g\,d\lambda$ is $\mu|_Y$-measurable. It follows that the left-hand side of (\ref{2measure}) is also defined.

Next we verify that the restriction of $B$ to $Y$ is a core map from $(Y,\mu|_Y)$ to $(S,\Sigma,\lambda)$. Using \ref{3c}, we see that $\mu|_Y$ is $\sigma$-finite. By \ref{3a}, the restriction of $B$ to $Y$ satisfies \ref{to2} of Definition \ref{coremap}; we verified \ref{me2} above; \ref{se2} is \ref{3c}; and \ref{fi2} follows from the definition of $Y$. Therefore (\ref{2measure}) is an abstract Hardy inequality.

Now suppose (\ref{3measure}) holds. Taking $h=\chi_{S_0}$, we have
\[
\int_Sh^p\,d\nu=0\quad\mbox{and}\quad\int_{B(z)}h\,d\omega=\nu^\perp(B(z)),
\]
so $\nu^\perp(B(z))=0$ $\mu$-almost everywhere on $Z$.

To show that (\ref{2measure}) holds, fix a $g\in L^+_\lambda=L^+_\omega$ and take $h=gu^{p'-1}\chi_{S\setminus S_0}$. Since $\nu(S_0)=0$,
\[
\int_{B(z)}g\,d\lambda=\int_{B(z)}gu^{p'}\,d\nu=\int_{B(z)}hu\,d\nu\le\int_{B(z)}h\,d\omega,
\]
for each $z\in Z$, and
\[
\int_Sh^p\,d\nu=\int_Sg^pu^{p'}\,d\nu=\int_Sg^p\,d\lambda.
\]
Since $Y\subseteq Z$, (\ref{2measure}) applied to $g$ follows from (\ref{3measure}) applied to $h$.

To show that $\mu(Z\setminus Y)=0$, apply the $\sigma$-finiteness of $\lambda$ to choose an increasing sequence of subsets $S_n\subseteq S$, each of finite $\lambda$-measure, such that $S=\cup_n S_n$. 

For each $n$,  $z'\mapsto\lambda(B(z')\cap S_n)$ is $\mu$-measurable. The range of $B$ is totally ordered. Therefore, for each $z\in Z$,
\[
\{z'\in Z: \lambda((B(z)\cap S_n)\setminus B(z'))=0\}=\{z'\in Z: \lambda(B(z')\cap S_n)\ge\lambda(B(z)\cap S_n)\},
\]
which is $\mu$-measurable. Taking the union over $n$ we see that 
\[
A(z)=\{z'\in Z:\lambda(B(z)\setminus B(z'))=0\}
\]
is $\mu$-measurable.

If $z\in Z\setminus Y$ then $\lambda(B(z)\cap S_n)\to\infty$ as $n\to\infty$. Taking $g=\chi_{B(z)\cap S_n}$ in (\ref{2measure}),
\[
\mu(A(z))^{1/q}\lambda(B(z)\cap S_n)=
\bigg(\int_{A(z)}\bigg(\int_{B(z')}g\,d\lambda\bigg)^q\,d\mu(z')\bigg)^{1/q}\le C\lambda(B(z)\cap S_n)^{1/p}.
\]
For sufficiently large $n$, $0<\lambda(B(z)\cap S_n)<\infty$ so we may divide through by $\lambda(B(z)\cap S_n)$ and let $n\to\infty$ to get $\mu(A(z))=0$.

For all $z\in Z$, $z\in A(z)$. By \ref{3d} there is a countable set $Z_0\subseteq Z\setminus Y$ such that 
\[
Z\setminus Y\subseteq\cup_{z\in Z\setminus Y}A(z)=\cup_{z\in Z_0}A(z).
\]
Thus $\mu(Z\setminus Y)\le\sum_{z\in Z_0}\mu(A(z))=0$.

For the converse, suppose $\nu^\perp(B(z))=0$ $\mu$-almost everywhere on $Z$, $\mu(Z\setminus Y)=0$ and (\ref{2measure}) holds. Fix $h\in L^+(\Sigma)$ and set $g=hu^{1-p'}\chi_{\{s\in S:u(s)>0\}}$. Then
\[
\int_Sg^p\,d\lambda=\int_Sg^pu^{p'}\,d\nu\le\int_Sh^p\,d\nu.
\]
For $\mu$-almost every $z$, 
\[
\int_{B(z)}h\,d\omega=\int_{B(z)}hu\,d\nu=\int_{B(z)}gu^{p'}\,d\nu=\int_{B(z)}g\,d\lambda.
\]
Since $\mu(Z\setminus Y)=0$, (\ref{3measure}) applied to $h$ follows from (\ref{2measure}) applied to $g$.
\end{proof}

Let us observe that the finiteness condition \ref{fi2} of Definition \ref{coremap} is not as restrictive as it may seem. Consider inequality (\ref{2measure}) for a map $B$ that satisfies \ref{to2}, \ref{me2} and \ref{se2}. Suppose that it also satisfies the very weak condition \ref{3d} of Theorem \ref{3meas}. Theorem \ref{3meas} then shows that the inequality cannot hold for a finite constant $C$ unless $\lambda(B(y))=0$ for $\mu$-almost every $y\in Y$. So either the inequality fails or we need only drop a set of $\mu$-measure zero for $B$ to also satisfy \ref{fi2}.

On our way to the proof of Theorem \ref{a2n}, we apply the results of Section \ref{S4} to establish a close relationship between the abstract Hardy operator $K_B$ and the intermediate operator $J_B:L^+\to L^+_\mu$, defined by 
\[
J_Bf(y)=\int_0^{\lambda(B(y))}f,\quad f\in L^+.
\]

\begin{theorem}\label{QandR} Let $(S,\Sigma,\lambda)$ and $(Y,\mu)$ be $\sigma$-finite measure spaces and suppose $B:Y\to \Sigma$ is a core map. Then there exist positive, admissible maps $R :L^1+L^\infty\to L^1_\lambda+L^\infty_\lambda$ and $Q :L^1_\lambda+L^\infty_\lambda\to L^1+L^\infty$, each of norm at most one,
such that $K_BR=J_B$ on $L^1+L^\infty$ and $J_BQ=K_B$ on $L^1_\lambda+L^\infty_\lambda$.
\end{theorem}
\begin{proof} The conclusion will follow from two applications of Theorem \ref{map}. First, we apply it with $(P,\mathcal P,\rho)=(S,\Sigma,\lambda)$, taking $(T,\mathcal T, \tau)$ to be Lebesgue measure on $(0,\infty)$. Let $\mathcal A=\{\emptyset\}\cup\{B(y):y\in Y\}$ and $r (A)=(0,\lambda(A))$. By \ref{to2}, \ref{se2} and \ref{fi2} we see that $\mathcal A$ is a $\sigma$-bounded ordered core of $(S,\Sigma,\lambda)$. If $A_1,A_2\in\mathcal A$ with $A_1\subseteq A_2$ then $(0,\lambda(A_1))\subseteq(0,\lambda(A_2))$, so $r $ is order-preserving, and the Lebesgue measure of $r (A_2)\setminus r (A_1)=[\lambda(A_1),\lambda(A_2))$ is $\lambda(A_2\setminus A_1)$, so $r $ is bounded with constant $c=1$. We conclude that there exists a map $R$ with the properties listed above.

Next, we apply Theorem \ref{map} with $(T,\mathcal T, \tau)=(S,\Sigma,\lambda)$, taking $(P,\mathcal P,\rho)$ to be Lebesgue measure on $(0,\infty)$. Let $\mathcal A=\{\emptyset\}\cup\{(0,\lambda(B(y))):y\in Y\}$ and note that by \ref{fi2}, $(0,\infty)\notin \mathcal A$. Clearly, $\mathcal A$ is totally ordered. By \ref{se2} there is a countable set $Y_0$ such that $\cup_{y\in Y}B(y)=\cup_{y\in Y_0}B(y)$ and it follows from \ref{to2} that 
\[
\cup\mathcal A=(0,\sup_{y\in Y}\lambda(B(y)))=(0,\sup_{y\in Y_0}\lambda(B(y)))
=\cup_{y\in Y_0}(0,\lambda(B(y)))
\]
so $\mathcal A$ is $\sigma$-bounded. 

Set $r (\emptyset)=\emptyset$ and for other $(0,t)\in\mathcal A$ set $r ((0,t))=B(y)$ for an arbitrarily chosen $y\in Y$ such that $\lambda(B(y))=t$. Suppose $(0,t_1)$ and $(0,t_2)$ are in $\mathcal A$ with $(0,t_1)\subsetneq(0,t_2)$. Clearly $0\le t_1<t_2$. Then $r ((0,t_2))=B(y_2)$ for a $y_2\in Y$ satisfying $\lambda(B(y_2))=t_2$. If $t_1=0$, $r ((0,t_1))=\emptyset\subseteq r ((0,t_2))$ and 
\[
\lambda(r ((0,t_2))\setminus r ((0,t_1)))=\lambda(B(y_2))=t_2=m((0,t_2)\setminus(0,t_1)).
\]
If $t_1>0$, then $r ((0,t_1))=B(y_1)$ for a $y_1\in Y$ satisfying $\lambda(B(y_1))=t_1$. Since $\lambda(B(y_1))<\lambda(B(y_2))$, \ref{to2} implies $r ((0,t_1))=B(y_1)\subsetneq B(y_2)=r ((0,t_2))$. Therefore
\[
\lambda(r ((0,t_2))\setminus r ((0,t_1)))=\lambda(B(y_2))-\lambda(B(y_1))=t_2-t_1=m((0,t_2)\setminus(0,t_1)).
\]
This shows that $r$ is order-preserving and bounded with constant $c=1$. To conclude that there exists a map $Q$ with the properties listed above, it remains only to point out that for any $y\in Y$, if $r(\lambda(B(y)))=B(y_1)$, then $\lambda(B(y_1))=\lambda(B(y))<\infty$ so, in view of \ref{to2}, $B(y)$ and $B(y_1)$ differ by a set of $\lambda$-measure zero.
\end{proof}

The above theorem, together with properties of the non-increasing rearrangement is enough to show that the operator norm of $K_B$ coincides with the operator norm of $H_b$, when $b=(\lambda\circ B)^*$.

\begin{corollary}\label{2.3I} Let $(S,\Sigma,\lambda)$ and $(Y,\mu)$ be $\sigma$-finite measure spaces and suppose $B:Y\to \Sigma$ is a core map. Set $b=(\lambda\circ B)^*$, where the rearrangement is taken with respect to the measure $\mu$. Then $b$ is a normal form parameter. Also, if $1\le p\le\infty$ and $0<q\le\infty$, then the operator norm of $K_B:L^p_\lambda\to L^q_\mu$ is $ N_{p,q}(b)$.
\end{corollary}

\begin{proof} Let $Q$ and $R$ be the maps of Theorem \ref{QandR}. The non-increasing rearrangement of any function is a normal form parameter so the first statement is immediate. Suppose $\|K_Bg\|_{L^q_\mu}\le C\|g\|_{L^p_\lambda}$ for all non-negative $g\in L^p_\lambda$ and fix a non-negative $f\in L^p$. Since the map $F$ defined by $F(\xi)=\int_0^\xi f$ is continuous, non-negative and non-decreasing, $[F\circ(\lambda\circ B)]^*=F\circ(\lambda\circ B)^*$. Thus $(J_Bf)^*=H_{(\lambda\circ B)^*}f=H_bf$. Also, since $R$ is admissible, $R:L^p\to L^p_\lambda$ with norm at most 1. Therefore,
\[
\|H_bf\|_q=\|(J_Bf)^*\|_q=\|J_Bf\|_{L^q_\mu}=\|K_B(R f)\|_{L^q_\mu}\le C\|R f\|_{L^p_\lambda}\le C\|f\|_p.
\]
Conversely, suppose $\|H_bf\|_q\le C\|f\|_p$ for all non-negative $f\in L^p$ and fix a non-negative $g\in L^p_\lambda$. Since $Q$ is a positive, admissible map, $Q:L^p_\lambda\to L^p$ with norm at most 1 and $Qg$ is non-negative so, as above, we have $(J_B(Q g))^*=H_b(Q g)$. Therefore
\[
\|K_B g\|_{L^q_\mu}=\|J_B(Qg)\|_{L^q_\mu}=\|(J_B(Qg))^*\|_q=\|H_b(Q g)\|_q\le C\|Q g\|_p\le C\|g\|_{L^p_\lambda}.
\]
We conclude that the operator norm of $K_B$ is equal to the operator norm of $H_b$, which is $ N_{p,q}(b)$ by definition. 
\end{proof} 

The previous corollary holds in considerably greater generality. Refer to \cite[Chapters 1, 2]{BS} for the definition of a \emph{rearrangement-invariant Banach function norm} $\theta$ and the space $X(\theta)$ it defines. All of our function norms are relative to Lebesgue measure on $(0,\infty)$, but we set $\theta_\lambda(f)=\theta(f^*)$ for $f\in L^+_\lambda$ so that $X(\theta_\lambda)$ becomes a space of $\lambda$-measurable functions.

In \cite[Chapter 3]{BS} it is shown that an admissible map $T:L^1_\mu+L^\infty_\mu\to L^1_\lambda+L^\infty_\lambda$ of norm at most one maps $X(\theta_\mu)$ to $X(\theta_\lambda)$ with norm at most one.

If the triangle inequality property is weakened to $\theta(f+g)\le c[\theta(f)+\theta(g)]$ for some constant $c$ we call the space {\it quasi-Banach}. 

\begin{corollary}\label{bfs} Let $(S,\Sigma,\lambda)$, $(Y,\mu)$, $B$ and $b$ be as in Theorem \ref{QandR}. If $\theta$ is a rearrangement-invariant Banach function norm and $\eta$ is a rearrangement-invariant quasi-Banach function norm, each relative to Lebesgue measure on $(0,\infty)$, then for each $C\ge0$, 
\[
\eta_\mu(K_Bg)\le C\theta_\lambda(g),
\]
for all $g\in L^+_\lambda\cap X(\theta_\lambda)$ if and only if 
\[
\eta(H_b f)\le C\theta(f)
\]
for all $f\in L^+\cap X(\theta)$.
\end{corollary} 
\begin{proof} Imitate the proof of Corollary \ref{2.3I}, replacing $\|\cdot\|_q$ and $\|\cdot\|_{L^q_\mu}$ by $\eta$ and $\eta_\mu$, respectively, and replacing $\|\cdot\|_p$ and $\|\cdot\|_{L^p_\lambda}$ by $\theta$ and $\theta_\lambda$, respectively. The positivity and admissibility of $R$ and $Q$ implies that $R:X(\theta) \to X(\theta_\lambda)$ and $Q:X(\theta_\lambda)\to X(\theta)$, each with norm at most 1.
\end{proof}

Although $J_B$ and $H_b$ are related by $(J_Bf)^*=H_bf$ for all $f\in L^+$, it is sometimes helpful to express their relationship at the level of operators. We will see that, under mild assumptions on $\lambda$, $B$, and $q$, we can find a positive, admissible $\Phi$ such that $J_B=\Phi H_b$. (Much stronger restrictions seem to be needed to get a $\Psi$ such that $H_b=\Psi J_B$. We will not pursue this further here except to say that for many core maps $B$ one can find a positive admissible $\Psi$.)

\begin{theorem}\label{Phi} Let $(S,\Sigma,\lambda)$ and $(Y,\mu)$ be $\sigma$-finite measure spaces, suppose $B:Y\to \Sigma$ is a core map and let $b=(\lambda\circ B)^*$. If $b(t)\to0$ as $t\to\infty$, then there exists a positive, admissible map $\Phi :L^1+L^\infty\to L^1_\mu+L^\infty_\mu$ of norm at most one,
such that $J_B=\Phi H_b$ on $L^1+L^\infty$.
\end{theorem}
\begin{proof} Our first step is to use the method of retracts to embed $(Y,\mu)$ in an atomless measure space $(\bar Y,\bar\mu)$. We may write $Y$ as the disjoint union $Y_0\cup(\cup_{j\in I} Y_j)$, where $Y_0$ is atomless and $Y_j$ is an atom for each $j$ in some index set $J$.
The $\sigma$-finiteness of $\mu$ implies that $J$ is countable and $\mu(Y_j)<\infty$ for each $j\in J$. If $g$ is $\mu$-measurable, then $g$ is constant $\mu$-a.e on each $Y_j$; we denote its constant value by $g(Y_j)$.

Let $\bar Y=Y_0\cup(\cup_{j\in I} I_j)$, a disjoint union, where each $I_j$ is an interval of length $\mu(Y_j)$, each in a different copy of $\mathbb R$ with Lebesgue measure. Define $\bar\mu$ on $\bar Y$ by $\bar\mu(E)=\mu(E\cap Y_0)+\sum_{j\in J}m(E\cap I_j)$. It is routine to verify that $(\bar Y,\bar \mu)$ is a $\sigma$-finite measure space and
\[
\mathcal Ig=\begin{cases} g,&\mbox{on }Y_0;\\g(Y_j)&\mbox{on }I_j, j\in J,\end{cases}
\]
defines a linear embedding of $L^1_{\mu}+L^\infty_{\mu}$ into $L^1_{\bar\mu}+L^\infty_{\bar\mu}$. Clearly, $\mathcal I$ is a positive, admissible map. The map
$$
\mathcal J\bar g=\begin{cases} \bar g,&\mbox{on }Y_0;\\
\frac1{|I_j|}\int_{I_j}\bar g&\mbox{on }Y_j\text{ for some }j\in J,\end{cases}
$$
defined for all $\bar g\in L^1_{\bar\mu}+L^\infty_{\bar\mu}$, is a one-sided inverse of $\mathcal I$ satisfying $\mathcal J\mathcal Ig=g$ for all $g\in L^1_{\mu}+L^\infty_{\mu}$.
Since averaging over the intervals $I_j$ does not increase either the $L^1_{\bar\mu}$ or $L^\infty_{\bar\mu}$ norms, $J$ is also a positive, admissible map.

Let $\Lambda=\lambda\circ B$. Clearly, $\mu_\Lambda=\bar\mu_{\mathcal I\Lambda}$ so $b=\Lambda^*=(\mathcal I\Lambda)^*$. By hypothesis, $(\mathcal I\Lambda)^*(t)\to0$ as $t\to\infty$. Since $\bar\mu$ is atomless, we may apply \cite[Corollary II.7.6]{BS} to get a measure-preserving (\cite[Definition II.7.1]{BS}) transformation $\sigma$ from the support of $\mathcal I\Lambda$ to the support of $b$ such that $\mathcal I\Lambda=b\circ\sigma$ on the support of $\mathcal I\Lambda$.

For each $h\in L^1+L^\infty$, let $h_0$ be the restriction of $h$ to the support of $b$ and let $\bar g=h_0\circ\sigma$ on the support of $\mathcal I\Lambda$ and zero elsewhere on $\bar Y$. Define $\Phi h=\mathcal J\bar g$. Clearly, $\Phi$ is a positive operator. Because $\mathcal J$ is admissible and \cite[Corollary II.7.2]{BS} shows that $|h_0\circ\sigma|=|h_0|\circ\sigma$ is equimeasurable with $|h_0|$ we have
\[
\|\Phi h\|_{L^1_\mu(Y)}=
\|\mathcal J \bar g\|_{L^1_\mu(Y)}
\le\|\bar g\|_{L^1_{\bar\mu}(\bar Y)}
=\|h_0\circ\sigma\|_{L^1_{\bar\mu}(\supp \mathcal I\Lambda)}
=\|h_0\|_{L^1(\supp b)}
\le\|h\|_1.
\]
Each of these steps remains valid with ``$1$'' replaced by ``$\infty$''. Thus, $\Phi$ is admissible and has norm at most one.

To prove $J_B=\Phi H_b$ we fix $f\in L^1+L^\infty$, define $F(\xi)=\int_0^\xi f$ and set $h=F\circ b$. Define $h_0$ and $\bar g$ as above. If $\bar y$ is in the support of $\mathcal I\Lambda$, then $b\circ\sigma(\bar y)=\mathcal I\Lambda(\bar y)$ and $\sigma(\bar y)$ is in the support of $b$, so $h\circ\sigma(\bar y)=h_0\circ\sigma(\bar y)$. Therefore,
\[
\bar g(\bar y)=h_0\circ\sigma(\bar y)=F\circ b\circ\sigma(\bar y)=F\circ(\mathcal I\Lambda)(\bar y).
\]
If $\bar y$ is not in the support of $\mathcal I\Lambda$,  we have $\bar g(\bar y)=0=F\circ(\mathcal I\Lambda)(\bar y)$. Thus $\bar g=F\circ(\mathcal I\Lambda)$. By the definition of $\mathcal I$ we see that $F\circ(\mathcal I\Lambda)=\mathcal I(F\circ\Lambda)$ so we have
\[
\Phi H_bf=\Phi h=\mathcal J\bar g=\mathcal J(F\circ(\mathcal I\Lambda))=\mathcal J\mathcal I(F\circ\Lambda)=F\circ\Lambda=J_Bf.
\]
This completes the proof.
\end{proof}

Let $1\le p\le\infty$ and $0< q\le\infty$ or, more generally, let $\theta$ be a rearrangement-invariant Banach function norm and $\eta$ be rearrangement-invariant quasi-Banach function norm, both relative to Lebesgue measure on $(0,\infty)$. The following diagrams commute:
\begin{equation}\label{cd}
\begin{tikzcd}[row sep=large, column sep=large]
L^p\arrow{r}{R} \arrow{rd}{J_B} \arrow{d}{H_b} & L^p_\lambda \arrow{d}{K_B} \arrow{r}{Q} & L^p  \arrow{ld}{J_B} \arrow{d}{H_b}\\
L^q & L^q_\mu\arrow[swap,dotted]{l}{\Psi}&L^q\arrow[swap,dashed]{l}{\Phi}
\end{tikzcd}\qquad
\begin{tikzcd}[row sep=large, column sep=large]
X(\theta)\arrow{r}{R} \arrow{rd}{J_B} \arrow{d}{H_b} & X(\theta_\lambda) \arrow{d}{K_B} \arrow{r}{Q} & X(\theta)  \arrow{ld}{J_B} \arrow{d}{H_b}\\
X(\eta) & X(\eta_\mu) \arrow[swap,dotted]{l}{\Psi}&X(\eta)\arrow[swap, dashed]{l}{\Phi}
\end{tikzcd}
\end{equation}

Dashed Arrows: The map $\Phi$, over the dashed arrows, only appears when $q\ge1$ and $b^*\to0$, or in the more general case, when $\eta$ is Banach and $b^*\to0$.

Dotted Arrows: For a $\mu$-measurable function $F$, $F^*$ cannot, in general, be expressed as the composition of $F$ with a measure-preserving transformation. See \cite[Example II.7.7]{BS}. However, for many functions $F$ there is a such a transformation. In particular, if there is a measure-preserving transformation $\sigma$ such that $b=(\lambda\circ B)^*=\lambda\circ B\circ\sigma$ the map $\Psi$, over the dotted arrows, will appear provided $q\ge1$ and $b^*\to0$, or in the more general case, when $\eta$ is Banach and $b^*\to0$.

Our final theorem should be viewed as an extension of Theorem \ref{cpctnormal}.
\begin{theorem}\label{acpct} Let $1<p\le q<\infty$, let $K_B$ be an abstract Hardy operator and let $b=(\lambda\circ B)^*$ be its normal form parameter. Then $K_B:L^p_\lambda\to L^q_\mu$ is compact if and only if $H_b:L^p\to L^q$ is compact.
\end{theorem} 
\begin{proof} Theorem \ref{QandR} provides maps $Q$, bounded from $L^p_\lambda\to L^p$, and $R$, bounded from $L^p\to L^p_\lambda$, such that $K_BR=J_B$ and $J_BQ=K_B$. It follows that $K_B$ is compact if and only if $J_B$ is compact.

If $H_b$ is compact then it is bounded so by Lemma \ref{niceb}, $b(t)\to0$ as $t\to\infty$. Thus Theorem \ref{Phi} provides a map $\Phi$, bounded from $L^q$ to $L^q_\mu$ such that $J_B=\Phi H_b$. This shows that $J_B$ is compact.

If $J_B$ is compact then $K_B$ is compact and hence bounded. By Corollary \ref{2.3I} $H_b$ is bounded and Lemma \ref{niceb} shows that $b(x)<\infty$ for all $x>0$ and $b(x)\to0$ as $x\to\infty$. Assume that $H_b$ is not compact. Then Theorem \ref{cpctnormal}\ref{cpct1} fails so either $\limsup_{y\to\infty}b(y)^{1/p'}y^{1/q}>0$ or $\limsup_{y\to0+}b(y)^{1/p'}y^{1/q}>0$. Now we apply Lemma \ref{cpctlm}, with $\Lambda=\lambda\circ B$ to get a sequence of unit vectors $f_{x_n}\in L^p$ such that $J_Bf_{x_n}$ has no convergent subsequence in $L^q_\mu$. Thus $J_B$ is not compact. This contradiction proves that $H_b$ is compact. 
\end{proof}

Theorem \ref{a2n} follows from Corollary \ref{2.3I} and Theorem \ref{acpct}.

%\bibliography{HardyHistory}{}
%\bibliographystyle{amsplain}
%\end{document}
\bibliographystyle{amsplain}
\providecommand{\bysame}{\leavevmode\hbox to3em{\hrulefill}\thinspace}
\providecommand{\MR}{\relax\ifhmode\unskip\space\fi MR }
\providecommand{\MRhref}[2]{%
  \href{http://www.ams.org/mathscinet-getitem?mr=#1}{#2}
}
\providecommand{\href}[2]{#2}

\end{document}